\documentclass[12pt]{amsart}

\usepackage{amsmath}
\usepackage{amssymb}
\usepackage[PostScript=dvips,noPS]{diagrams}
\usepackage{graphicx}
\usepackage{float,color}

\newtheorem{theorem}{Theorem}[section]
\newtheorem{lemma}[theorem]{Lemma}

\newtheorem{corollary}[theorem]{Corollary}

\newtheorem{definition}[theorem]{Definition}

\newtheorem{remark}[theorem]{Remark}

\let\oldtocsection=\tocsection

\let\oldtocsubsection=\tocsubsection

\let\oldtocsubsubsection=\tocsubsubsection

\renewcommand{\tocsection}[2]{\hspace{0em}\oldtocsection{#1}{#2}}
\renewcommand{\tocsubsection}[2]{\hspace{1em}\oldtocsubsection{#1}{#2}}
\renewcommand{\tocsubsubsection}[2]{\hspace{2em}\oldtocsubsubsection{#1}{#2}}

\DeclareMathOperator{\im}{Im}
\DeclareMathOperator{\re}{Re}

\DeclareMathOperator{\coker}{coker}
\DeclareMathOperator{\Arg}{Arg}

\newcommand{\CC}{\mathbb C}
\newcommand{\HH}{\mathbb H}
\newcommand{\RR}{\mathbb R}
\newcommand{\ZZ}{\mathbb Z}
\newcommand{\FF}{{\mathbb{F}_2}}
\newcommand{\FFF}{{\mathbb Z}/2}

\newcommand{\circaccent}{\mathaccent"7017 }
\newcommand{\vv}{\check}
\newcommand{\qb}{\bar}

\newcommand{\SU}{\mathrm{SU}(2)}
\newcommand{\SO}{\mathrm{PSU}(2)}
\newcommand{\SUn}{\mathrm{SU}(n)}
\newcommand{\RP}{{\mathbb{P}^3(\RR)}}
\newcommand{\CP}{{\mathbb{P}^3(\CC)}}
\newcommand{\RPt}{{\mathbb{P}^2(\RR)}}
\newcommand{\CPt}{{\mathbb{P}^2(\CC)}}

\DeclareMathOperator{\tr}{tr}
\DeclareMathOperator{\Ad}{Ad}
\newcommand{\G}{G}
\newcommand{\M}{M}
\newcommand{\su}{\mathfrak{su}(2)}
\newcommand{\g}{\mathfrak{g}}

\newcommand{\sgn}{{\rm sgn}}

\newcommand{\pt}{\mathord{*}}

\newcommand{\Trefor}{{\mathfrak T}}
\newcommand{\calS}{{\mathcal S}}

\newarrow{Equalto}{=}{=}{=}{=}{=}

\newcommand{\wwedge}{\vee}
\newcommand{\Wwedge}{\bigvee}
\newcommand{\Sq}{{\rm Sq}}

\newcommand{\kew}{p}
\newcommand{\comm}{\nu}
\newcommand{\bfl}{{[[}}
\newcommand{\bfr}{{]]}}

\newcommand{\leftbfl}{\left[\left[}
\newcommand{\rightbfr}{\right]\right]}

\title{Flat Connections and the Commutator Map for $\SU$}

\author{Nan-Kuo Ho}
\address{Department of Mathematics, National Tsing Hua University,
Hsinchu 300, and National Center for Theoretical Sciences, Taipei 106,
Taiwan}
\email{nankuo@math.nthu.edu.tw}

\author{Lisa C. Jeffrey}
\address{Department of Mathematics, University of Toronto, Toronto,
Canada}
\email{jeffrey@math.toronto.edu}

\author{Paul Selick}
\address{Department of Mathematics, University of Toronto, Toronto,
Canada}
\email{selick@math.toronto.edu}

\author{Eugene Z. Xia}
\address{Department of Mathematics, National Cheng Kung University,
Tainan 70101, Taiwan}
\email{ezxia@ncku.edu.tw}
\thanks{
N.H. was supported by a Ministry of Science and Technology (Taiwan) grant 
107-2115-M-007-016-MY2. 
L.J. and P.S were supported by NSERC (Canada)  Discovery Grants.
E.X. was supported by  Ministry of Science and Technology (Taiwan) grants
107-2115-
-006-009 and 108-2115-M-006-008.}

\begin{document}

\begin{abstract}
We study the topology of the $SU(2)$-representation variety of the 
compact oriented surface of genus $2$ with one boundary component about which
the holonomy is a generator of the center of $SU(2)$.
\end{abstract}

\maketitle
\tableofcontents

\section{Introduction}

In a landmark 1983 paper \cite{AB}, 
Atiyah and Bott studied the moduli
space of gauge
equivalence classes of flat $\SUn$ connections on an
orientable surface
 with one
fixed boundary component about which the holonomy was a generator in
 the center of $\SUn$. These objects had
been studied since the work of Narasimhan and 
Seshadri \cite{Narasimhan-Seshadri}.  Using the
Morse theory of the function which associates the norm square of the
curvature
to a connection, they obtained the group structure of the cohomology of
the moduli space.
Later  Thaddeus \cite{Thaddeus} identified the ring structure of the
cohomology when $n=2$.
Then  Witten \cite{witten} identified  the ring structure of the
cohomology for general $n$  using methods from
quantum field theory, which were subsequently made rigorous by Jeffrey
and Kirwan \cite{arbrank}.

More precisely, when $n=2$ for example, the space studied by Atiyah and
Bott is the moduli space of
projectively flat connections on a vector bundle of rank $2$ and degree
$1$ on an oriented closed 2-manifold $\Sigma$ of genus
$g$. This moduli space can be identified via \cite{FS92} with the moduli
space of
flat $\SU$ connections $A$ on a trivial $\SU$ bundle
over an oriented surface $\Sigma^0$ of genus $g$ with $1$ boundary
component about which the holonomy is $-I$. The correspondence arises by
considering the union of $\Sigma^0$
with  a $2$-disk $D$. We equip $\Sigma^0$ with the
connection $A$,  and equip $D$ with a non-flat connection
which also has holonomy
$-I$ around the boundary. The resulting (non-flat) connection  on the
closed surface  $\Sigma^0 \cup D$  is a connection on a bundle with
degree $1$ and
rank $2$ over $\Sigma$.

Atiyah-Bott's paper restricts to genus $g \ge 2$. In this paper, we
consider the case
$g = 2$ and $n=2$,  and show that the cohomology ring can
be identified  using much more elementary methods, namely
 a particular Mayer-Vietoris sequence. This is the first purely
 topological proof of this result.

In \cite{HJNX}, the first, second and fourth author studied
Hamiltonian torus actions defined on the space of gauge
equivalence classes of flat connections  on a closed
orientable $2$-manifold of genus~$2$.
This space had been identified with $\CP$ by Narasimhan and
Ramanan \cite{Narasimhan-Ramanan} using algebraic-geometric methods.
The image of the moment map for this torus action is a tetrahedron,
like the image of the moment map for the standard torus action on~$\CP$.
The authors of \cite{HJNX} attempted to obtain this identification using
the methods of
toric geometry, but were ultimately unsuccessful.

This paper should be
 accessible to a graduate student who has completed a
first course in algebraic topology.
Aside from arguments from first principles, the main tool used is the
Mayer-Vietoris sequence.
Some basic knowledge of fibre bundles is needed,
such as
classifying maps, transition functions, and for the later parts of the
paper
the Serre spectral sequence.

The reader will also need to know the Bockstein homomophism.
We briefly describe here its role in the paper.
If the spaces involved have torsion, the K\"unneth Theorem and
Serre Spectral Sequence are very messy when using a coefficient ring
such as $\ZZ$ which is not a field.
The calculations become much simpler when we reduce to $\ZZ/2$
coeffients.
Unfortunately, while the (co)homology with integer coefficients
determines (co)homology with other coefficients through the
Universal Coefficient Theorem, the process throws away information.
For a finite $CW$-complex, the (co)homology with integer coefficients
can be recovered from that with field coefficients provided one also
knows the action of the Bocksteins.
Information about the Bockstein and its properties can be found, for
example, in~\cite{Mosher-Tangora} pages 22 and~61.
Therefore, in several of the calculations in this paper we calculate
first with
$\ZZ/2$ coefficients including Bockstein action and then use that
to obtain the result with integer coefficients.

Finally, at the end of subsection  \ref{ss:cohvm},  for
completeness we give a calculation of some cup products using an
argument that requires more advanced homotopy theory including the use
of
Steenrod operations.
This calculation is not critical to the paper, and the reader can skip
it without losing the overall thrust of the paper.

Let $G=\SU$ and let $T$ be its maximal torus of diagonal matrices.
We often write elements of~$\SU$ as quaternions $z+wj\in\HH$ where
$z,w\in\CC$
and $|z|^2 + |w|^2 = 1$.
Let $\comm:G \times G \to G,\ (g,h)\mapsto [g,h]$ denote the commutator
map.
Note that its only singular value is the identity matrix~$I$.

Our results are as follows.
\begin{enumerate}
\item For regular values $e^{i \theta} \ne I$ of~$\comm$, we
describe a $T$-equivariant homeomorphism $\comm^{-1} (e^{i \theta})
\to
\RP$, where $T$ acts by the diagonal conjugation action on $G \times G$ and left
translation on~$\RP \cong \SO$.
\item We identify the cohomology ring and a cell-decomposition of a
    space homotopy equivalent to
the space $\Trefor := \comm^{-1}(I)$ of commuting pairs.
\item We compute the cohomology of $M:= \mu^{-1}(-I)$, where
$\mu: G^4 \to G$ is the product of commutators.
\item We give new calculations of the cohomology of $A:= M/G$, both as
groups  and as rings.  The group structure is due to Atiyah and Bott
\cite{AB}. Atiyah and Bott's research program was
motivated by Morse theory, but completed by algebraic geometry.
The ring structure of $A$ is due to Thaddeus~\cite{Thaddeus}, who used
methods
from algebraic geometry and results from conformal field theory.
\item We calculate the cohomology of the total space of the
prequantum line bundle over~$A$.
\item We  also identify the transition functions of the induced
$\SO$ bundle $M \to A$.
\end{enumerate}

The existence of a homeomorphism in results (1) was known, but not known
to be $T$-equivariant.  In (2), the cohomology groups and homotopy type
of {\em the suspension of $\Trefor$} was known, but not $\Trefor$
itself. Result (4) is a new (purely topological) proof of the known
results. Results (3), (5), (6) are completely new.

The layout of the paper is as follows.
In \S \ref{commmap}, we examine properties of the commutator
map~$\comm$. For $\theta \in [0, \pi]$, we introduce
$X_\theta :=\comm^{-1}(e^{i \theta})$, the space of pairs in $G\times G$
whose
commutator equals~$e^{i \theta}$. For the regular values, corresponding
to $\theta > 0$, in \S\ref{explicithom} we construct an explicit
$T$-equivariant
homeomorphism from $X_\theta$ to ~$\RP$.
Following this, in \S \ref{retraction} we construct an explicit
retraction
from
the complement of $X_\pi$ to the space $\Trefor:=X_0= \comm^{-1}(I)$
of commuting pairs. This allows us to use a Mayer-Vietoris sequence
to identify the cohomology ring and a cell decomposition of~$\Trefor$.
In \S\ref{Atiyah} we study the space $A= M/G$. We describe analogous
retractions to the one in \S \ref{retraction} which allow us to write
$A$ as the union of two open sets each of which is homotopy equivalent
to~$\Trefor$.
Using the resulting Mayer-Vietoris sequence we compute the cohomology
groups~$H^*(A)$.
In \S\ref{qAtiyah} we discuss the prequantum line bundle of~$A$.
Again using the Mayer-Vietoris sequence, we calculate the ring of the total space of a
related line bundle which allows us to obtain the ring structure
of~$H^*(A)$.
In \S\ref{nineman} we describe the $9$-manifold~$M$.
We show that the bundle $M\to A$ has a local trivialization over two
open
sets, and find the transition function.
Using the Mayer-Vietoris sequence, we calculate the cohomology ring of the
restriction
of the bundle to a certain submanifold of~$A$ and later use this to
compute~$H^*(M)$.
In \S\ref{wall} we recall results of Wall on $6$-manifolds
(see \cite{Wall}) and show that the cohomology we  computed
for~$A$ is consistent with the results in that paper.
Finally, in \S\ref{cohM} we conclude the calculation
of~$H^*(M)$ and the cohomology of the prequantum line bundle over~$A$.

Notational note: To avoid conflict with $[~,~]$ which will denote
``commutator'' we use $\bfl X \bfr$ to denote ``the equivalence class
of~$X$''. Throughout the paper, we used  $\cong$ for homeomorphism,
$\simeq$ for homotopy equivalence, and $:=$ for definition.

\section{The commutator map on $\SU\times \SU$}\label{commmap}
\subsection{Preliminaries}

The group $\SO$ is isomorphic to $\SU/\{\pm I\}$.  Let 
$$
\comm: \SU \times \SU \to \SU, (g,h)\mapsto [g,h]=ghg^{-1}h^{-1}
$$
be the
commutator map.
For $\theta\in[0,\pi]$, set:
$$W_\theta:=\{(g,h)\in \SU\times \SU \mid [g,h]\sim e^{i\theta}\}$$
where ``$\sim$" means ``conjugate in ~$\SU$'';
$$X_\theta:=\{(g,h)\in \SU\times \SU \mid [g,h]=e^{i\theta}\};$$
$$Y_\theta:=
\{g\in \SU \mid \exists\, h\in \SU \mbox{ with } [g,h]=e^{i\theta}\}.$$
For $S\subset [0,\pi]$ set
$W_S:=\cup_{\theta\in S} W_\theta$ and similarly define $X_S$ and $Y_S$.
In particular, $W_{[0,\pi]}=\SU\times \SU$ and
$X_{[0,\pi]}=\comm^{-1}(T)$.

In this section we consider $X_\theta$ for $\theta\in(0,\pi]$.

\begin{lemma}(Meinrenken \cite{Meinrenken})\label{Eckhard}
$X_\theta\cong \RP$ for $\theta\in(0,\pi]$.
\end{lemma}

\begin{proof}
Since the only singular value of the commutator map
$\comm$ is the identity matrix $I$,
$\comm^{-1}(g)$ is homeomorphic to $\comm^{-1}(h)$ provided neither $g$
nor $h$ is~$I$.
In particular, $X_\theta:=\comm^{-1}(e^{i\theta})\cong
\comm^{-1}(-I)=X_\pi$
if $\theta\ne 0$.
Pick a basepoint~$\pt\in X_\pi$.
Define $\Phi:\SU\to X_\pi$ by $\Phi(g):=g\pt g^{-1}$.
This gives a transitive action of $\SU$ on~$X_\pi$ for which the
stabilizer of each element is the center $\{I,-I\}$ of~$\SU$.
Thus $X_\pi\cong \SO\cong \RP$.
\end{proof}

Since the above Lemma~\ref{Eckhard} gives an explicit homeomorphism $X_\theta\to \RP$
only in the case $\theta=\pi$, in the next section we will proceed
differently
and give an explicit homeomorphism for arbitrary $\theta\in(0,\pi]$.
Moreover, such homeomorphism can be chosen to be
$T$-equivariant.
In preparation, we begin by giving some properties of the commutator map
in this section.

\begin{lemma}
Let $\theta\in(0,\pi]$.
Suppose that $g=z+wj$ belongs to $Y_\theta$.
Then $z=\pm |z|e^{i\theta/2}$.
\end{lemma}

\begin{proof}
Write $z=|z|e^{i\tau}$ in polar form.
Then there exists $h$ such that $[g,h]=e^{i\theta}$.
Then $hgh^{-1}=e^{-i\theta}g$.
Taking traces of both sides gives $\cos(\tau)=\cos(\tau-\theta)$.
Thus $\tau=\pm(\tau-\theta)$ (as elements of $S^1\cong\RR/(2\pi\ZZ)$).
Since $\theta\ne0$, $\tau=-(\tau-\theta)$, so $2\tau=\theta$.
That is, $\tau=\theta/2$ or $\theta/2+\pi$.
Therefore $z=|z|e^{i\tau}=\pm |z|e^{i\theta/2}$.
\end{proof}

Similarly if $h\in \SU$ such that there exists $g\in \SU$ with
$[g,h]=e^{i\theta}$, then $[h,g]=e^{-i\theta}$ so $h$ has the
form $h=Qe^{-i\theta/2}+wj$ for some $Q\in[-1,1]$, where $w\in\CC$
with $|w| = \sqrt{1-Q^2}$.

\begin{corollary}\label{Xthetaform}
Suppose $(g,h)\in X_\theta$ where $\theta\in(0,\pi]$.
Then $$(g,h)=(Pe^{i\theta/2}+Raj,Qe^{-i\theta/2}+Sbj)$$
where $P, Q\in[-1,1]$, $R=\sqrt{1-P^2}$, $S=\sqrt{1-Q^2}$,
$a,b\in S^1$.
\end{corollary}

\begin{corollary}\label{Ytheta}
If $\theta\ne0$ then $Y_\theta\cong S^2$.
\end{corollary}

\begin{proof}
Let $g=Pe^{i\theta/2}+Raj\in Y_\theta$ where $P\in[-1,1]$,
$R=\sqrt{1-P^2}$, $a=e^{i\alpha}\in S^1$.
A homeomorphism is given by $g\mapsto (Ra,P)$ regarding the
right hand side  as an element of $S^2$ written in cylindrical
coordinates.
That is,
$g\mapsto \bigl(R\cos(\alpha), R\sin(\alpha),P\bigr)\in S^2\subset
\RR^3$.
\end{proof}

Suppose $(Pe^{i\theta/2}+Raj,Qe^{-i\theta/2}+Sbj)\in X_\theta$
with $\theta\ne0$.
Since $\dim X_\theta=\dim \RP=3$ there must be some relation
among $P$, $Q$, $a$, $b$.
\begin{lemma}[{\bf Canonical Relation}] \label{canonlem}
Let $(Pe^{i\theta/2}+Raj,Qe^{-i\theta/2}+Sbj)\in X_\theta$
with $\theta\ne0$.
Then
\begin{equation}\label{canoneq}
PQ-vRS=e^{i\theta}(PQ-\bar{v}RS)
\end{equation}
where $$v:=a\bar{b}. $$
\end{lemma}

\begin{proof}
Expand
$[Pe^{i\theta/2}+Raj,Qe^{-i\theta/2}+Sbj]=e^{i\theta}$
and equate either the coefficients of~$1$ or of $j$.
(Either results in the same equation.)
\end{proof}

Write $v=e^{i\phi}$.
This defines $\phi$ as an element of $\RR/(2 \pi \ZZ)$.
For $P\ne 0$, solving the canonical relation for~$Q$
(recalling that $S=\sqrt{1-Q^2}$) gives
\begin{lemma}
Suppose $\theta\ne0$.
For $P\ne0$,
\begin{equation}
Q=\sgn(P) \frac{KR}{\sqrt{P^2+K^2R^2}}
\end{equation}
where $K(\phi):=\cos(\phi)-\cot(\theta/2)\sin(\phi)$.
\end{lemma}

\begin{proof}
If $R=0$ then since $\theta\ne0$, the canonical equation gives $Q=0$
and the Lemma is satisfied.
So assume $R\ne0$.

Since we assumed $P\ne 0$, the canonical relation gives $S\ne0$.
Set $K:=PQ/(RS)$.
Recalling that $S=\sqrt{1-Q^2}$, we have
$$Q=\sgn(P) \frac{KR}{\sqrt{P^2+K^2R^2}}.$$
The canonical relation gives
$$K(1-e^{i\theta})=v-e^{i\theta}\bar{v}$$
or equivalently
$$K(1-e^{i\theta})(e^{-i\theta/2})=e^{-i\theta/2}(v-e^{i\theta}\bar{v}).$$
Since
$$(1-e^{i\theta})e^{-i\theta/2}=e^{-i\theta/2}-e^{i\theta/2}
=(-2i)\sin(\theta/2),$$
we get 
$$\sin(\theta/2)K=e^{-i\theta/2}(v-e^{i\theta}\bar{v})/(-2i)
=e^{-i\theta/2}(e^{i\theta}e^{-i\phi}-e^{i\phi})/(2i)
$$
$$ \phantom{\sin(\theta/2)K=}
=\bigl(e^{i(\theta/2-\phi)}-e^{-i(\theta/2-\phi)}\bigr)/(2i)
=\sin(\theta/2-\phi).$$
Therefore
$$K = \csc(\theta/2) \sin(\theta/2-\phi)
=\cos(\phi)-\cot(\theta/2)\sin\phi.$$
\end{proof}

\begin{theorem}\label{canonfunction}
For $\theta\ne0$ and  $0<|P|<1$, we have
$$Q=
\sgn(P/K)\frac{1}{\sqrt{1+\frac{P^2\sin^2(\theta/2)}{R^2\sin^2(\theta/2-\phi)}}}.$$
When $|P|=1$, we have  $R=0 $ and the equation reduces to $Q=0$.
\end{theorem}

\begin{proof}
For $\theta\ne0$, $|P|\in(0,1)$, equation~(\ref{canoneq})
is equivalent to
$$
Q=\sgn(P) \frac{KR}{\sqrt{P^2+K^2R^2}}
=\sgn(P/K)\frac{1}{\sqrt{1+\frac{P^2}{K^2R^2}}},
$$
where
$$\sin(\theta/2)K =\sin(\theta/2-\phi).$$
Therefore
$$
Q=\sgn(P/K)\frac{1}{\sqrt{1+\frac{P^2}{K^2R^2}}}
=\sgn(P/K)\frac{1}{\sqrt{1+\frac{P^2\sin^2(\theta/2)}{R^2\sin^2(\theta/2-\phi)}}}.
$$
\end{proof}

If $|P| = 1$, then $R = 0 $ and the equation implies $Q=0$.
Similarly $|Q|=1$ implies $P=0$; however $P=0$ does not imply
$|Q|=1$. Instead, for $P=0$  we have

\begin{lemma}\label{p=0}
If $\theta\ne0$ then
$$\{h\mid[aj,h]=e^{i \theta}\} = \{Qe^{i\theta/2}+Sbj \mid
{\mbox{either\ }S=0\mbox{\ or\ }a^2=b^2e^{i\theta}}\}.$$
\end{lemma}

\begin{proof}
If $P=0$ then $R=1$.
Substituting these into the canonical relation (Lemma~\ref{canonlem}) gives
$vS=e^{i\theta}\bar{v}S$.
Thus either $S=0$ or $v=e^{i\theta} \bar{v}$.
Since $|v|=1$, we have $\bar{v}=v^{-1}$ so the latter condition says
$v^2=e^{i\theta}$, or equivalently $a^2=b^2e^{i\theta}$.
\end{proof}

\subsection{Waves in the $(\phi, Q)$ cylinder}\label{waves}

Consider a fixed $P$ with $|P|\in (0,1]$.
For each $\theta>0$ we get a periodic function
$$Q_{\theta,P}(\phi)
=\frac{\sgn(P/K)}
{\sqrt{1+\frac{P^2\sin^2(\theta/2)}{R^2\sin^2(\theta/2-\phi)}}},$$
where by convention we set $Q_{\theta,1}(\phi)\equiv 0,$
which is the limiting function as $|P|\to1$.

\begin{definition}
We define  a {\em wave} as a continuous map from $S^1 $ to $ S^1
\times[-1,1]$.
\end{definition}
See Figures 1 and 2. 

Let $\Gamma_{\theta,P}$ be one period of the graph of $Q_{\theta,P}$ (in
the $(\phi,Q)$ cylinder) which we draw as a periodic function in
the $(\phi,Q)$ cylinder.
The curve $\Gamma_{\theta,P}$ (in the $(\phi,Q)$ cylinder) looks
somewhat like a
sine wave.
As $\theta\to 0$ or $P \to 0$,  the curves $\Gamma_{\theta,P}$ approach
a square wave.

Extend the definition of $\Gamma_{\theta,P}$ to the case where
$\theta=0$
by letting it equal the limiting square wave.
As $P\to0$ we also get a square wave but the resulting wave depends
on whether $P$ approaches $0$ from the left or the right.
So we define $\Gamma_{\theta,0^+}$ as the limiting square wave as
$P\to0^+$ and similarly we have its reflection~$\Gamma_{\theta,0^-}$.

The first plot shows three waves with $P=1/\sqrt{2}$ illustrating
$\theta=1.2$,
(the dotted graph),
$\theta=0.5$ (the dashed graph) and the limiting wave $\theta=0$
(the solid line).
\begin{figure}
\centerline{\includegraphics{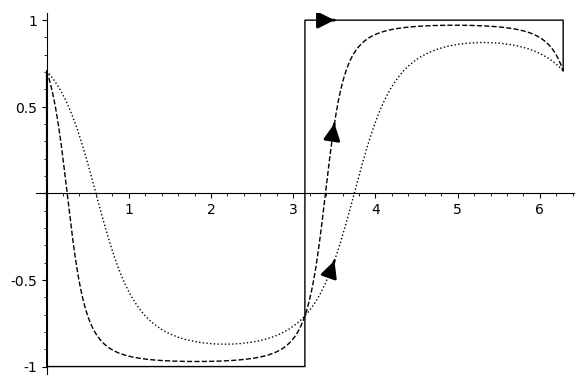}}
\caption{Three waves with $P = 1/\sqrt{2}$}
\end{figure}
The second  plot illustrates four waves with $\theta=\pi$.
The dotted graph, nearest the axis, illustrates $P=0.99$.
At $P=1$ it would become the axis.
The dashed graph shows $P=0.5$ and its reflection, the dash-dotted graph
shows $P=-.5$.
The solid line is the square wave~$\Gamma_{\pi,0^+}$.
The wave~$\Gamma_{\pi,0^-}$ (not shown) is its reflection about the
axis.
\begin{figure}
\centerline{\includegraphics{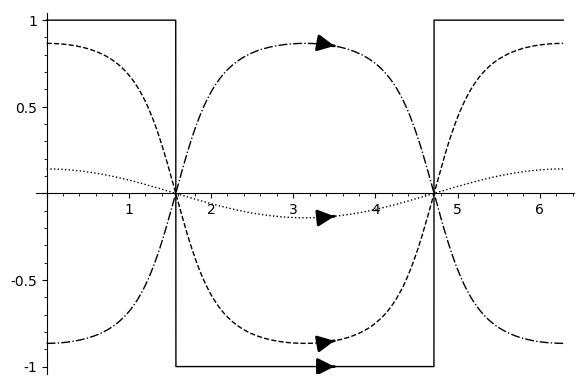}}
\caption{Four waves with $\theta = \pi$}
\end{figure}

We have oriented the curves counterclockwise if $P>0$ and clockwise if
$P<0$.
Taking the limit, we orient $\Gamma_{\theta,0^+}$ in the
counterclockwise
direction and $\Gamma_{\theta,0^-}$ in the clockwise direction.

\section{Explicit homeomorphism $X_\theta\cong \RP$}\label{explicithom}

\begin{theorem}\label{thm:explicithom}
There is a $T$-equivariant homeomorphism
$\Phi_\theta: X_\theta\cong \RP$, for all $\theta\in (0,\pi]$.
\end{theorem}
Recall that $T$ acts on $X_\theta$ by diagonal conjugation and on $\RP \cong \SO$ by left multiplication.
\begin{proof}
First, we will give an explicit homeomorphism $\Phi_\theta:X_\theta\cong
\RP$.
As per Corollary~\ref{Xthetaform}, the equation
$[x,y]=e^{i\theta}$ requires that when written as
 quaternions, $x$, $y$ have the form $x=Pe^{i\theta/2}+Raj$,
$y=Qe^{-i\theta/2}+Sbj$,
where $P$, $Q$ $\in [-1,1]$, $R=\sqrt{1-P^2}$, $a\in S^1$,
$S=\sqrt{1-Q^2}$, $b\in S^1$.
Furthermore, Equation (\ref{canoneq}) of Lemma~\ref{canonlem} must be
satisfied,
which determines $Q$ in terms of $P,a,b$.
Thus generically the points of $X_\theta$ are parametrized by the values
of $P$, $a$, and~$b$ or equivalently by $P$, $a$, and~$\phi$, where
$\phi \in \RR/(2\pi \ZZ)$ is defined by  $b = a e^{i \phi}$.

Generically a point in $\RP$ is also determined by three parameters.
After picking a representative $z+wj$ with $|z|^2+|w|^2=1$, the point
$[[z+wj]]\in\RP$ is determined by the absolute value and argument of the
ratio $z/w$ together with the argument of~$z^2$, recall that we are
using  the notation~$[[X]]$ for the equivalence class of $X$  introduced
in the note at the end of \S 1.
Replacing the representative by $-z-wj$ does not affect these
parameters.

Recall that the space $X_\theta$ is generically parametrized by the
parameters
$$P,R, a, Q, S, b$$
with  $P,Q,R,S \in \RR$ satisfying
$P^2 + R^2  = 1$ and  	 $Q^2 + S^2 = 1$, and
$a, b \in \CC$ with
$|a| = |b| = 1,$  and satisfying the canonical
relation (\ref{canonlem}).
In order to
 define the inverse map $$\Phi_\theta^{-1}:\RP\to X_\theta,$$
we need to specify
$P,Ra, Q, Sb$ as functions of the parameters  $z$ and~$ w$.

Let $\bfl z+wj\bfr$ represent an element of $\RP$ where $z,w\in \CC$
with $|z|^2+|w|^2=1.$
Write $z=Ze^{i\zeta}$, $w=We^{i\omega}$,
and $\delta:=\zeta-\omega$,
where $\zeta$ is well defined only if $Z\ne0$, $\omega$ is well defined
only when $W\ne0$
and $\delta$ is well defined only when $ZW\ne0$.

\medskip

\noindent{\em Step 1: Definition of $P, Ra $ } 

We define
\begin{equation}
\label{pdef}
P(z,w) :=|z|^2-|w|^2=Z^2-W^2 = 2Z^2 -1
\end{equation}
 and
\begin{equation} \label{radef}  (Ra)(z,w) :=-2e^{i\theta/2}zw,
\end{equation}
where $P^2 + R^2 = 1$
and $|a| = 1$.
The value of $a$ is uniquely determined only if $zw\ne0$.
This concludes Step 1.

In Step 1 we   showed how $P, Ra$ are specified as functions of  $z$
and~$w$.
We now show how to determine $Q$ and $Sb$.

\medskip

\noindent{\em Step 2: Definition of $Q, Sb$}

As in \S \ref{waves},
the parameters $P$ and $\theta$ determine a wave $\Gamma_{\theta,P}$ in
the
$(\phi,Q)$ cylinder  $C:=S^1 \times [-1,1]$.
All the waves pass though the points $(\theta/2,0)$ and
$(\theta/2+\pi,0)$.
Let $\pi:C\to S^2$ be the map which collapses $S^1\times\{1\}$ to a point (the
north pole)
and $S^1\times\{-1\}$ to another point (the south pole):
\begin{equation} \pi(\phi,Q)
=(\sqrt{1-Q^2}e^{i\phi},Q), \end{equation}
 where the right hand side
describes a point in $ S^2$ written in cylindrical coordinates.
Notice that although the two  curves $\Gamma_{\theta,0^+}$ and
 $\Gamma_{\theta,0^-}$ in~$ C$ differ (the latter is the reflection of
 the former), their images
in~$S^2$ under~$\pi$ are the same. Indeed,  each becomes the great
circle
passing through the poles and the antipodal points $\pi(\theta/2,0)$
and~$\pi(\theta/2+\pi,0)$.

\medskip
\noindent{\em Step 2, Case I: $zw \ne 0 $}

In case I, the quantity $a$ is well defined because $zw  \ne 0 $.
Let
\begin{equation} \label{sdef}
s(z,w) := \Arg(z/w).
\end{equation}
The quantity $s$ is also well defined  because $zw \ne 0 $.
Notice that $zw=0$ if and only if $|P|=1$.
Therefore in case I we have $0  \le |P| < 1. $

\medskip

Recall that $b$ is specified in terms of the parameter $\phi$ by
$b = a e^{i \phi}$, so we only need to determine $Q$ and $\phi$.
Let
\begin{equation}  \label{HHHHH}
\gamma_{\theta,P}:[0,2\pi]\to S^2 \end{equation} be the parametrization
of the curve
$\pi(\Gamma_{\theta,P})$ by normalized arc length, measured from
$\gamma_{\theta,P}(0):=\pi(\theta/2,0)$.
To make this precise, we need to decide whether increasing $s$ moves
clockwise
or counterclockwise around the curve. Recall that in \S 2.2
we chose to orient the curve counterclockwise if $P>0$ and clockwise if
$P<0$.
Taking the limit, we oriented $\pi(\Gamma_{\theta,0^+})$ in the
counterclockwise
direction and $\pi(\Gamma_{\theta,0^-})$ in the clockwise direction to
obtain
corresponding parametrizations $\gamma_{\theta,0^+}$ and
$\gamma_{\theta,0^-}$.
(See \S \ref{waves} for the definition of $\Gamma_{\theta,P}$.)

In the case $0 < |P| < 1$,
define $\phi$ and $Q$ by letting $(\sqrt{1-Q^2}e^{i \phi}, Q)$
to be the cylindrical coordinates of the point $\gamma_{\theta,P}(s) \in
\pi (\Gamma_{\theta,P}). $
Extend the definition of $(\phi,Q)$ to the case $P=0$ by using
either  $\gamma_{\theta,0^+}(s)$ or equivalently
$\gamma_{\theta,0^-}(s)$.
These produce the same point since the curves are related by reflection
about the
equator and we have reversed orientation.

Set\begin{equation} \label{Sdef} S(z,w) :=\sqrt{1-Q^2}. \end{equation}
Set \begin{equation}
\label{bdef} b(z,w) :=ae^{-i\phi}. \end{equation}
For $zw\ne0$, define $\Phi_\theta^{-1}: \RP \to X_\theta$  by
\begin{equation} \label{phithdef}
\Phi_\theta^{-1}(\bfl z+wj\bfr):=
(Pe^{i\theta/2}+Raj,Qe^{-i\theta/2}+Sbj)
\end{equation}
with the  values of $P$, $Ra$, $Q$, $Sb$ given by
(\ref{pdef}),  (\ref{radef}),  (\ref{Sdef})
and (\ref{bdef}) above.
By construction, the canonical relation will be satisfied, so the
right hand side will lie in~$X_\theta$.
This concludes Step 2, Case I.

\medskip

\noindent{\em Step 2, Case II:  $zw=0$.}

Suppose alternatively that  $zw=0 $.
Then
either  $Z=0$ (which results in $P=-1$, $R=0$)
or $W=0$ (which results in $P=1$, $R=0$).
Since $|P|=1$, we set $Q=0$, $S=1$ . Then the
canonical relation (Lemma \ref{canonlem}) is satisfied.

In Case II, formula (\ref{radef}) does not determine a unique $a$ and
formula (\ref{sdef}) does not define a value of $s$. Thus we cannot
argue as in Case I.
To complete the definition of $\Phi_\theta^{-1}$ in case II,
it remains to determine a value for $b$ which makes the
map $\Phi_{\theta}^{-1} $  continuous.

Since $zw=0$, either $z=0$ or $w=0$. First, suppose $z=0$.
  Then we do not have a unique $\zeta$, but there is a well defined
  $\omega$.
  Consider a point with $z$ near zero but not equal to zero.
This results in a point   with $R$ near $0$ and $P$ near~$-1$.
Thus the graph $\Gamma_{\theta,P}$ is close to the graph of the
constant map $Q=0$
(and since we are near the equator $Q=0$ of the sphere $S^2$, the
projection
$\pi$ has little effect on distances).
In general, the relation between changes in $s$ on the resulting value
of~$\phi$ is complicated.
However on the graph $Q=0$, distances between points correspond exactly
to
changes in $\phi$.
When $z$ is near zero, the wave is close to the straight line
$Q=0$, so a change in $s$ produces the almost identical change
in~$\phi$.
Thus when $z$ is near zero, we have
$$\alpha-\beta \approx \phi\approx s = \Arg(z/w)=\zeta-\omega,$$
where $a = e^{i\alpha} $ and $b = e^{i \beta}$.
We wish to determine the value of $\beta$ which makes
$b := e^{i \beta}$ a  continuous function of $z$ and $w$.
 The definition of $Ra$ gives $\alpha-\theta/2=\zeta+\omega+\pi$.
Therefore
$$			\beta\approx
\alpha-\zeta+\omega=\zeta+\omega+\theta/2-\zeta+\omega+\pi
=2\omega+\theta/2+\pi. $$
  So when $z=0$ we set
\begin{equation} \label{betadef1} \beta:=2\omega+\theta/2+\pi
\end{equation} to obtain a continuous map.

Similarly when $w=0$, we have $P=1$ and for $w$ near zero, we have
$\phi\approx -s$.
Therefore we set
\begin{equation} \label{betadef2} \beta:=2\zeta+\theta/2+\pi
\end{equation} in this case to produce a
continuous map.

In other words, in Case II we again define $\Phi_\theta^{-1}$ as in
(\ref{phithdef})
with the values of $P, Ra$  given by (\ref{pdef}) and (\ref{radef}) with
$b$ given by  (\ref{betadef1})
or (\ref{betadef2}).
This concludes Step 2, case II.

\medskip
Notice that since $|z|^2$, $|w|^2$ and $zw$ are quadratic, replacing
the representative $z+wj$ by $-z-wj$ yields the same point.
Therefore the map $\Phi^{-1}_\theta$ is well defined.
This concludes the definition of the map
$\Phi_\theta^{-1}$.
It gives a homeomorphism from $\RP$ to $X_{\theta}$.

If $t$ lies in the maximal torus~$T$, then replacing $[[z+wj]]$ by
 $t[[z+wj]]$ produces no change in $P$ and $s$ and thus no change in
$Q$ or~$\phi$.
However $a$ is replaced by $at^2$ and $b$ is replaced by $bt^2$.
Thus $\Phi_\theta^{-1}$ is $T$-equivariant with respect to
translation action on $\RP$ and conjugation on $X_\theta$.

\end{proof}

\section{Retraction}\label{retraction}

The only singular value of the commutator map $\comm:\SU\times \SU\to
\SU$
is the identity matrix $I$.
The space $\Trefor:=X_0$ (which is not a
manifold) of commuting $2$-tuples in $\SU$ has been studied in several
places (see for example
\cite{Adem-Cohen},\cite{Baird-Jeffrey-Selick},\cite{Crabb}).
See also \cite{Bazett} for commuting 2-tuples in $SU(2)$. 

Although we will provide explicit deformation retractions when needed,
for motivation we note the following theorem which shows the existence
of such retractions under some hypotheses.
\begin{theorem}\label{MilnorMorse}
Let $M$ be a smooth manifold and let  $f:M\to\RR$ be smooth.
Let $c$ be an isolated critical value of~$f$.
If $f$ has no critical values in $(c,d)$ then
$f^{-1}(c)$ is a deformation retract of $f^{-1}([c,d))$.
\end{theorem}

\begin{proof}
The result follows from the Remark~3.4 of~\cite{MiMorse} and the
discussion in Chapter~$3$
of that book.
\end{proof}

Applying Theorem \ref{MilnorMorse} to the function
${\rm Trace}:X_{[0,\pi)}\to\RR$
shows that there is a deformation retraction $X_{[0,\pi)}\to \Trefor$.
In this section we will describe two such deformation
retractions $X_{[0,\pi)}\to \Trefor$ with a view towards computing
transition functions for a bundle $M\to A$ to be defined later.

\subsection{Gradient Flow}\label{gradientflow}

Let $\g := \su$, the Lie algebra of $G:=\SU$. Recall that $\g \subseteq
\HH$ consists of the pure imaginary quaternions while $G \subseteq \HH$
consists of the unit quaternions. Since $G$ is compact, it has a
$G$-invariant Riemannian metric defined by the trace form $B$ and the
associated norm.  In quaternion notation,
$$
B(u,v) = - 2\re(uv), \ \ \  ||u||^2 = B(u,u), \ \ \ u,v \in \g.
$$
where $\re(w)$ denotes the real part of $w$.
This induces a $G$-invariant metric on $G^2$.
Recall
\begin{equation}\label{eq:dC}
\comm : G^2 \rTo G, \ \ \ \comm(g,h) = [g,h] = ghg^{-1}h^{-1}.
\end{equation}
Any vector in $T_{(g,h)}G^2$ can be identified with a vector in $\g^2$
by left translation and we often perform our computation in $\g$.  For
$u \in \g$ and $g \in G$, let
$$
u^g = \Ad(g)(u).
$$
A direct computation shows that
$$
d\comm_{(g,h)} : \g^2 \rTo \g, \ \ \ d\comm_{(g,h)}(u,v) = (u^{h^{-1}} -
u + v - v^{g^{-1}})^{hg}.
$$
Notice that the map
$$
\g \rTo \g, \ \ \ u \mapsto u^h - u
$$
has rank $2$ if and only if $h \neq \pm I$.  Hence  $d\comm_{(g,h)}$ has
full rank (rank $3$) if and only if $[g,h] \neq I$;  that is,  the only
singular value of the map $\nu$ is $I$.
Now we let
$$
f : G^2 \rTo \RR, \ \ \ f = 2 \re \circ \comm,
$$
where $\re(a)$ means the real part of $a$.

For $(u,v) \in \g^2$, the gradient $\nabla f$ at $(g,h) \in \G$, by
definition, satisfies
\begin{equation} \label{eq:metric}
B(\nabla f, (u,v)) = {\tr} \left (\comm(g,h) (d\comm_{(g,h)}(u,v))\right ).
\end{equation}
\begin{lemma}
Suppose $A \in G$.  Then $\re(Au) = 0$ for all $u \in \g$ if and only if
$A = \pm I$.
\end{lemma}
\begin{proof}
This follows from direct computation.  Let
$$
A = a_0 + a_1 i + a_2 j + a_3 k, \ \  \ u = u_1 i + u_2 j + u_3 k, \ \ \
|A| = 1.
$$
Then
$$
\re(Au) = -\sum_{l = 1}^3 a_l u_l.
$$
Hence $\re(Au) = 0$ for all $u \in \g$ if and only if $a_l = 0$ for $l =
1,2,3$ if and only if $A = \pm I$.
$A = \pm I$.
\end{proof}
It then follows from Equations (\ref{eq:dC}, \ref{eq:metric}) that
\begin{corollary} \label{cor:gradient}
$\nabla f = 0$ if and only if $\comm(g,h) = \pm I$.
\end{corollary}
For $c \in [-2,2]$, let $M_{[a,b]} = f^{-1}([a,b])$ and $M_c =
f^{-1}(c).$  Since $G^2$ is compact, $M_c$ is compact since it is the
inverse image of a closed set.  Fix $d \in (-2,2)$ and let
$$
F : G^2 \times \RR \rTo G^2
$$
be the gradient flow of $f$ such that $F(g,h,0) \in M_d$.
It then follows from Corollary \ref{cor:gradient}  that
\begin{corollary} \label{cor:flow to d}
Let $(g,h) \in M_{c}$ for $c \in (-2,2)$.  Then there exists $(g',h')
\in M_d$ and $t$ such that
$F(g',h',t) = (g,h)$.
\end{corollary}
\begin{proof}
Suppose $c < d$.  Choose $a$ such that $d < a < 2$.

Since $||\nabla f_{(g,h)}|| > 0$ for all $(g,h) \in M_{[c,a]}$, the
space $M_{[c,a]}$ diffeomorphically retracts to $M_{[d,a]}$ by the
$(-\nabla f)$-flow by Morse theory.  Moreover $M_c$ is diffeomorphic to
$M_d$ by this flow.   The case of $c < d$ is similar, but using the
$(\nabla f)$-flow instead.
\end{proof}
Starting at $(g,h) \in M_c$ with $c \neq \pm 2$, and again by Corollary
\ref{cor:gradient}, we have
$$
 \lim_{t \to -\infty} F(g,h,t) \in M_{-2},  \ \ \  \lim_{t \to \infty}
 F(g,h,t) \in M_2.
$$
Since $f$ is real algebraic, by $\L$ojasiewicz inequality, the gradient
flow converges \cite{Loja}. Here the limit is  pointwise with respect to
the metric $B$, but it is also uniform since $G^2$ is compact. Let
$$
M_{-\infty} := \lim_{t \to -\infty} F(M_d,t), \ \ \ M_{\infty} :=
\lim_{t \to \infty} F(M_d,t).
$$
It is immediate that both $M_{-\infty} \subseteq M_{-2}$ and $M_{\infty}
\subseteq \M_2$ are compact since both are images of the compact set
$M_d$ by the flow $F$.
\begin{theorem}
$M_{\infty} = M_2$.
\end{theorem}
\begin{proof}
$M_2$ is a proper subvariety in $G^2$ and $G^2 \setminus M_2$ is open
and dense in $G^2$.
Let $(a,b) \in M_2$, $\epsilon > 0$ and $D_\epsilon(a,b) \subset G^2$
the $\epsilon$-ball centered at $(a,b)$.  Then $D_\epsilon(a,b)
\setminus M_2$ is open and dense in $D_\epsilon(a,b)$.  Since $\lim_{t
\to \infty} ||\nabla f|_{F(g,h,t)}|| = 0$ for $(g,t) \not\in M_{-2}$,
there exists $0 < \delta < \epsilon$ and $(g,h) \in D_\epsilon(a,b) \cap
M_{2-\delta}$ such that $\lim_{t \to \infty} F(g,h,t) \in
D_\epsilon(a,b) \cap M_2$.  By Corollary \ref{cor:flow to d}, there
exists $(g',h') \in M_d$ and  $t$ such that $$F(g',h',t) =(g,h).$$
  Hence $\lim_{t \to \infty} F(g',h',t) \in D_\epsilon(a,b) \cap M_2.$
  Hence
$\lim_{t \to \infty} F(M_d,t)$ is dense in $M_2$.  Since $M_d$ is
compact, $\lim_{t \to \infty} F(M_d,t)$ is closed in $M_2$.  Hence
$\lim_{t \to \infty} F(M_d,t) = M_2. $
\end{proof}

By definition, $M_2=f^{-1}(2)=\nu^{-1}(I)=X_0$. Hence, we obtain a
deformation retraction $X_{[0,\pi)}\to \Trefor:=X_0$.

\subsection{Explicit Retraction}\label{explicitretraction}
In this subsection, we consider another
retraction $r:X_{[0,\pi)}\to \Trefor:=X_0$.
Unlike the gradient flow retraction, the restriction of $r$ to
$X_\theta$
 is not onto for $\theta >0$.
However the map $r$ is explicit enough that we can use it to compute
transition functions.

With the notations of Section 2, we set $C:=S^1\times[-1,1]$ and
$\phi$ was defined by $v=e^{i\phi}$ where $v:=a\bar{b}$.
Recall that
for each pair $P$, $\theta$ with $|P|>0$, the
curve $\Gamma_{\theta,P}\subset C$ passes through $(\theta/2,0)$,
while for $P=0$, the curves $\Gamma_{\theta,0^+}$ and
$\Gamma_{\theta,0^-}$ pass through $(\theta/2,0)$.

For $\theta\in(0,\pi]$ and $a\in S^1$, each value of $P\ne0$
determines a
unique wave $\Gamma_{\theta,P}$ and each point $(\phi,Q)$ on that wave
determines a point
$(Pe^{i\theta/2}+\sqrt{1-P^2}aj,Qe^{-i\theta/2}+\sqrt{1-Q^2}ae^{-i\phi}j)
\in X_\theta$.

For $|P|>0$ we let $\tilde{\gamma}_{\theta,P}:[0,2\pi]\to C$
be the parametrization of the curve $\Gamma_{\theta,P}$ by
normalized arclength from $(\theta/2,0)$ measured counterclockwise if
$P>0$
and clockwise if $P<0$.
Similarly in the case $P=0$ we have parametrizations
$\Gamma_{\theta,0^+}$
and $\Gamma_{\theta,0^-}$ oriented consistently with the preceding.

Define a homeomorphism
$\Psi_{\theta,\theta',P}:\Gamma_{\theta,P} \cong \Gamma_{\theta',P}$
by
$$\Psi_{\theta,\theta',P}\bigl(\tilde{\gamma}_{\theta,P}(s)\bigr):=
\tilde{\gamma}_{\theta',P}(s).$$
That is, it takes a point on the curve $\Gamma_{\theta,P}$ to the
corresponding point on the curve $\Gamma_{\theta',P}$ preserving the
ratio
$$\frac{\mbox{arc length of segment to point}}{\mbox{arc length of
curve}}.$$

For $(\phi,Q)\in\Gamma_{\theta,P}$
let $(\phi_{\theta,\theta', P},Q_{\theta,\theta',P})$ be the coordinates
(in the $(\phi,Q)$ cylinder) of
$\Phi_{\theta,\theta',P}(\phi,Q)\in\Gamma_{\theta',P}$.
For $t\in[0,1]$, set
$$(\phi_{t\theta,P},Q_{t\theta,P}):=\Psi_{\theta,t\theta,P}(\phi,Q), \ \ \ 
S_{t\theta,P}:=\sqrt{1-Q_{t\theta}^2}.$$
We define a deformation
$ r:X_{(0,\pi]}\times [0,1] \to X_{[0,\pi]}$ with
$\im r_0\subset X_0=\Trefor$, where for a homotopy
$H$ we use the convention $H_t(x):=H(x,t)$.

First consider $P\ne0$.  In this case we set
\begin{equation}\label{definitionr}
r\bigl((Pe^{i \theta/2}+Raj,Qe^{-i\theta/2}+Sbj),t\bigr):=
(Pe^{it\theta/2}+Raj, Q_{t\theta,P}e^{-it\theta/2}
+S_{t\theta,P}b_{t\theta,P}j).
\end{equation}
By construction, $\im r_0\subset X_0=\Trefor$.

\begin{lemma}
The map $r$ is continuous and extends (uniquely) to a continuous
function
$$r:~\overline{X_{(0,\pi]}}\times [0,1]\to \overline{X_{(0,\pi]}}$$
with  $r\bigl(\overline{X_{(0,\pi]}},0\bigr)\subset X_0=\Trefor$.
\end{lemma}
Here $\bar{B}$ means the closure of $B$.
\begin{proof}
Set
$$B:=\{(Pe^{i \theta/2}+Raj,Qe^{-i\theta/2}+Sbj)\in X_{(0,\pi]}\mid
P\ne0\}.$$
Note that the closure $\bar{B}$ equals the closure
$\overline{X_{(0,\pi]}}$.
By writing an explicit formula for $r$ in terms of integrals according to the
arclengths, we see that  the function $r$ is uniformly continuous on
the domain $B\times [0,1]$.
Therefore $r$ extends uniquely to~$\overline{X_{(0,\pi]}}$.
\end{proof}

\begin{remark}
$\overline{X_{(0,\pi)}}$ is a proper subset
of~$X_{[0,\pi]}$.
Its intersection with $X_0$ contains only points of the form
$(P+Raj,Q+Sbj)$.
Points $(g,h)$ where $g$, $h$ lie in the same maximal torus but do not
have
the above form are not in $\overline{X_{(0,\pi)}}\cap X_0$.
Although $W_0=X_0$, these points lie in $\overline{W_{(0,\pi)}}$ but not
in $\overline{X_{(0,\pi)}}$.
\end{remark}

We extend the domain of $r$ to $X_{[0,\pi]}\times [0,1]$ by
setting $(r_t)|_{X_0}:=\mbox{identity}$ for all~$t$ and thus obtain a
deformation
\begin{equation}\label{definitionextendr}
r:~X_{[0,\pi]}\times [0,1]\to X_{[0,\pi]}
\end{equation}
giving $X_{[0,\pi]} \simeq X_0$.

\begin{lemma}
For each $t$, the map $r_t:X_\theta\to X_{t\theta}$ is $T$-equivariant.
\end{lemma}

\begin{proof}
The $T$-action by $u\in T$ replaces $a$ and $b$ by $u^2a$, $u^2b$
respectively
with no change in $\phi$ or the other parameters.
Since the graphs $\Gamma_{t\theta,P}$ are unaffected by this action, the
resulting point $\phi_{t\theta,P}$ is unaffected, so the resulting
$(u^2a)e^{i\phi_{t\theta,P}}$ has been multiplied by $u^2$.
\end{proof}
Summarizing, we have
\begin{theorem}
$r:X_{[0,\pi]}\times [0,1]\to X_{[0,\pi]}$ is  $T$-equivariant,
giving a $T$-equivariant strong deformation retraction $X_{[0,\pi]} \to
X_0$.
\end{theorem}

Similarly we have a $T$-deformation
$r':\overline{X_{(0,\pi]}}\times [0,1]\to \overline{X_{(0,\pi]}}$
with $\im(r'_0)\subset X_\pi$  obtained by extending the map
$r':X_{(0,\pi]}\times [0,1]\to X_{(0,\pi]}$ given by
\begin{eqnarray}\label{definitionr'}
\lefteqn{r'\bigl((Pe^{i \theta/2}+Raj,Qe^{-i\theta/2}+Sbj),t\bigr):=}\\
&&(Pe^{i\bigl(t\theta+(1-t)\pi\bigr)/2}+Raj,
Q_{t\theta+(1-t)\pi}e^{-i\bigl(t\theta+(1-t)\pi\bigr)/2}
+S_{t\theta+(1-t)\pi}b_{t\theta+(1-t)\pi}j) \nonumber
\end{eqnarray}

By abuse
of notation
we omit the subscript $P$ from  $Q_{\theta,P}$ and its
analogues $S_{\theta,P}$ and $b_{\theta,P}$,  writing
$Y_\theta$ instead of $Y_{\theta,P}$.
We also define deformations
$r_W:W_{[0,\pi)}\times [0,1]\to W_{[0,\pi)}$ and
$r_W':W_{(0,\pi]}\times [0,1]\to W_{(0,\pi]}$
by
\begin{equation}\label{definitionrw}
r_W(x',y',t):=gr(x,y,t)g^{-1}\quad
\mbox{and}\quad r_W'(x',y',t):=gr'(x,y,t)g^{-1}
\end{equation}
respectively, where for $(x',y')\in W_\theta$ we find
$g\in G/T$ and $(x,y)\in X_\theta$ such that $[x',y']=g[x,y]g^{-1}$.
Here it was necessary to restrict the definition of $r_W$
to $W_{[0,\pi)}$ since $e^{i\pi}$ is central and so for
$(x,y)\in W_\pi$ we do not get a well defined $g\in G/T$.
Similarly the domain of $r_W'$ is restricted to $W_{(0,\pi]}$.
The maps $r_W$ and $r_W'$ are well defined since $r$ and $r'$
are $T$-equivariant.

\subsection{Cohomology of $\Trefor$}
{\ }
\medskip

Let $\FF$ denote the field of $2$ elements and $\ZZ/2$ denote
the group $\ZZ/2\ZZ$. We use different notations when we want to
emphasize the field structure or the group structure.

The homotopy type of the suspension $\Sigma\Trefor$ has been given
in several papers (see for example
\cite{Adem-Cohen} and \cite{Baird-Jeffrey-Selick}),  along with the
group structure of $H^*(\Trefor)$.
Here we obtain the homotopy type of $\Trefor$.

Let $\kew:\SU\to \SU/{\pm I}= \SO \cong \RP$ be the canonical
projection.
Denote by $T$ the maximal torus of diagonal matrices in~$\SU$
and $\bar{T}=\kew(T)$ the corresponding maximal torus
of~$\RP$.
As topological groups, $T=\bar{T}=S^1$ and $\kew:T\to\bar{T}$
corresponds
to the squaring map.
Recall that $W_\theta:=\{(x,y)\in \SU\times \SU\mid [x,y]\sim
e^{i\theta}\} $
(where ``$\sim$" means ``conjugate").
Then $\SU\times \SU=\cup_{\theta\in[0,\pi]} W_\theta$.
Notice that $W_0=X_0$ and $W_\pi=X_\pi$ since $I$ and $-I$ are in the
center of~$\SU$.

\begin{lemma} \label{diagaction}
For $\theta\in (0,\pi)$, $W_\theta\cong \bigl(X_\theta \times
\RP\bigr)/T$,  where $T$ acts
diagonally by conjugation on the first factor and left translation
by $\bar{T}=\kew(T)$ on the second.
After applying  the $T$-homeomorphism~$\Phi_\theta$, this can be
stated
equivalently as
$W_\theta\cong \bigl(\RP \times \RP\bigr)/\bar{T}$ where $\bar{T}$
acts
diagonally.
\end{lemma}

\begin{proof}
Let $(x',y')$ belong to $W_\theta$.
Then there exists $[g]\in \RP/\bar{T}$ such that
$g[x',y']g^{-1}=e^{i\theta}$.
Set $x:=gx'g^{-1}$, $y:=gy'g^{-1}$.
Then $(x,y)\in X_\theta$.
Define a homeomorphism
\begin{equation} W_\theta\cong \bigl(X_\theta \times \RP\bigr)/T
\end{equation}
by $(x',y')\mapsto [[\bigl((x,y),g \bigr)]]. $
\end{proof}

\begin{corollary}
For $\theta\in (0,\pi)$, $W_\theta\cong \RP \times S^2$.
\end{corollary}

\begin{proof}
Consider the fibration
$$\RP\to (\RP\times \RP)/\bar{T} \rTo^{\pi_2} \RP/\bar{T}$$
This fibration has a retraction given by $(g,h)\mapsto g^{-1}h$
(which is well defined).
Thus $(\RP\times \RP)/\bar{T} \cong \RP\times S^2 \cong \RP \times
(\RP/\bar{T}). $
\end{proof}

Note that the above homeomorphism is not canonical.
For example, we could have projected onto the first factor
instead of the second.

Write $\SU\times \SU=U\cup V$ where
$U:=W_{[0,\pi)}$,
$V:=W_{(0,\pi]}$.
Then $r_W$ and $r'_W$ give strong deformation retractions
$U\to W_0=X_0=\Trefor$ and $V \to W_{\pi}=X_{\pi}\cong \RP$.
Noting that $\RP/\bar{T}\cong S^2$, we can now recover the cohomology of
$\Trefor:=X_0$ from the Mayer-Vietoris sequence for $\SU\times \SU=U\cup
V$.

\begin{lemma} \label{cohomologytrefor}
The integral cohomology of~$\Trefor$ is
$$H^q(\Trefor)=
\begin{cases}
\ZZ&q=0,2;\cr
\ZZ\oplus\ZZ&q=3;\cr
\ZZ/2&q=4;\cr
0&\mbox{otherwise}.\cr
\end{cases}
$$
All cup products in $\tilde{H}^*(\Trefor)$ are zero.
\end{lemma}

The computation of the group structure of $H^*(\Trefor)$ has been done
in
several places in various ways.
This group structure is given in \cite{Adem-Cohen} and
\cite{Baird-Jeffrey-Selick}, although the
methods of those papers do not obtain the ring structure.
See also \cite{Bazett} for the cohomology groups of the space of 
commuting 2-tuples in $SU(2)$.

\begin{proof}
We calculate the groups $H^*(\Trefor)$ by means of the Mayer-Vietoris
sequence for $\SU\times\SU=U\cup V$.
We have homotopy equivalences $U\simeq W_0=X_0=\Trefor$ and
$V\simeq W_\pi=X_\pi\cong\RP$.
Set
$$B:=U\cap V=W_{(0,\pi)}\simeq \RP\times S^2.$$

As explained in the introduction, first we calculate the cohomology with
$\FF$ coefficients.\\
With these coefficients,
$$H^*(V)=H^*(\RP)=\langle 1, v, v^2, v^3\rangle, \text{ where } \ \ |v|=1;$$
$$H^*(U\cup V)=H^*(\SU\times\SU)=\langle 1, w, w', ww'\rangle, \ \  \text{ where }
|w|=|w'|=3;$$
$$H^*(B)=\langle 1, t, s, t^2, st, t^3 , st^2, st^3 \rangle, \ \  \text{ where }
|s|=2, \ t=|1|,$$
and
the Bockstein is  
$$\beta(v)=v^2, \ \ \beta(t)=t^2,\ \ 
\beta(w)=\beta(w')=\beta(s)=0.$$
See for example \cite{Mosher-Tangora} (pp. 22 and 61) for further
information about the Bockstein.

Let $j $ denote the inclusion $ B \to U$,
and let $j'$ denote the inclusion $B \to V$.
Let $i$ denote the inclusion $U\to \SU\times \SU$.

$$\displaylines{
0\rTo H^1(U)\oplus \underbrace{H^1(V)}_\FF \rTo \underbrace{H^1(B)}_\FF
\rTo\cr
\underbrace{H^2 (\SU\times \SU)}_0 \rTo
H^2 (U) \oplus \underbrace{H^2 (V)}_\FF
 \rTo \underbrace{H^2 (B)}_{\FF+\FF}\rTo\cr
\underbrace{H^3 (\SU\times \SU)}_{\FF+\FF} \rTo
H^3 (U) \oplus \underbrace{H^3 (V)}_\FF
 \rTo \underbrace{H^3 (B)}_{\FF+\FF}\rTo\cr
\underbrace{H^4 (\SU\times \SU)}_0 \rTo
H^4 (U) \oplus \underbrace{H^4 (V)}_0 \rTo \underbrace{H^4
(B)}_\FF\rTo\cr
\underbrace{H^5 (\SU\times \SU)}_0
\rTo H^5 (U) \oplus \underbrace{H^5 (V)}_0 \rTo \underbrace{H^5 (B)}_\FF
\rTo\cr
\underbrace{H^6 (\SU\times \SU)}_\FF \rTo  0
}$$
Exactness gives us:
$$H^1(U)=0; \ \ H^2(U)=\FF; \ \ H^3(U)=\FF+\FF+\FF; \ \ H^4(U)=\FF; \ \ 
H^5(U)=0$$
(Here we use $+$ to distinguish from $\oplus$.)
Let $a$, $b$, $s_1$, $s_2$, $c$ denote the generators in degrees $2$,
$3$,
$3$, $3$, $4$ of $H^*(U)$ respectively.
Here $s_1:=i^*(w)$; $s_2:=i^*(w')$, so $\beta(s_1)=\beta(s_2)=0$,
where $\beta$ is the Bockstein.

In the above, the isomorphisms
\[H^2(B)\cong H^2(\RP\times S^2)=\FF+\FF, \quad \mbox{and} \quad
H^3(B)\cong H^3(\RP\times S^2)=\FF+\FF\]
depend upon a choice of
homotopy equivalence $B\simeq \RP\times S^2$, which, as noted earlier,
involves a choice.
With a suitable choice, we have $j^*(a)=s$ and $j^*(b)=st$
from  which we conclude $\beta(a)=0$ and $\beta(b)=c$.
Therefore the integral cohomology groups are as stated.

For degree reasons, the only possible nonzero cup product in
$\tilde{H}^*(\Trefor;\FF)$ is~$a^2$.
Since $j^*(a^2)=\bigl(j^*(a)\bigr)^2=s^2=0$ and $j^*$ is an isomorphism
on $H^4(~;\FF)$, we see that $a^2=0$.
\end{proof}

Let $P^n(2)$ denote the Moore space
$P^n(2):=S^{n-1}\cup_2 e^n=\Sigma^{n-2} \RPt$, where $e^n$ denotes an
$n$-cell and $\cup_2$ denotes a degree~$2$ attaching map.

\begin{corollary}\label{CWTrefor}
$\Trefor\simeq S^2\Wwedge(\wwedge^2 S^3)\Wwedge P^4(2) $. In particular,
$\Trefor$ is simply connected.
\end{corollary}

The fact that this holds after one suspension is given in several
references,
such as~\cite{Baird-Jeffrey-Selick}.

\begin{proof}
As above, write $\SU\times \SU=U\cup V$ where $U \simeq\Trefor$,
$V\simeq\RP$ and $U\cap V\simeq \RP\times S^2$. Applying Van Kampen's
theorem shows that $\Trefor$ is simply connected.

Up to homotopy equivalence, the only way to build a simply connected
$CW$-complex for $\Trefor$ with the cohomology groups  as in Lemma
\ref{cohomologytrefor} is
the homotopy cofibre of an attaching map
$f:P^3(2)\to  S^2\Wwedge(\wwedge^2 S^3)$ which induces the homomorphism
$f^*=0$
on~$\tilde{H}^*(~)$.
The fact that $f^*$ is zero on $\tilde{H}^*(~)$ requires that the
composition
$P^3(2)\rTo^f  S^2\Wwedge(\wwedge^2 S^3)\to \wwedge^2 S^3$
to be null, so  $f$ is determined by the composition
$$P^3(2)\rTo^f  S^2\Wwedge(\wwedge^2 S^3)\to S^2$$ which we denote
as~$\bar{f}$.
Our calculation of the ring structure of $H^*(\Trefor;\FF)$ in
Lemma~\ref{cohomologytrefor}  gives that
the cup products in the reduced cohomology $\tilde{H}^*(~;\ZZ)$ of the
homotopy cofibre of $\bar{f}$ are zero.

The group of homotopy classes of maps $[P^3(2),S^2]$ equals $\ZZ/2$ with
the
nonzero element given by $P^3(2)\rTo^{\rm pinch}S^3\rTo^{\eta}S^2$
where $\eta$ denotes the Hopf map.
The homotopy cofibre of $\eta$ is $\CPt$ which has a nonzero cup
square on $H^2(~;\FF)$.  Naturality gives that the homotopy cofibre of
$P^3(2)\rTo^{\rm pinch}S^3\rTo^{\eta}S^2$ has a nonzero cup square
on $H^2(~;\FF)$, so $\bar{f}$ is not the nonzero element of
$[P^3(2),S^2]$.
Thus $\bar{f}$ is null homotopic and therefore $f$ is null homotopic.
Hence $\Trefor$ is the wedge, as claimed.
\end{proof}

\subsection{Centralizers}\label{centralizers}

If $g\in \SU$ is not $\pm I$, then its centralizer $Z_{\SU}(g)$ forms a
maximal
torus in~$\SU$. It can be described explicitly as follows.

The Lie algebra~$\g:=\su$ can be identified with the space of pure
imaginary
quaternions~$\{\xi=x i+y j+z k,\  x,y,z\in \RR\}\subset \HH$.
The exponential map $\exp:\g\to\SU$ is given by
$\exp(\xi)=\cos(|\xi|)+\sin(|\xi|)\xi/|\xi|$, where
$|\xi|^2:=x^2+y^2+z^2$.
Let $\circaccent{\g}=\{\xi\in\su\mid 0< | \xi | < \pi\}$.
The exponential map is a homeomorphism $\exp:\circaccent{\g}\cong
\SU\setminus\{\pm I\}$.
For $g\in \SU\setminus\{\pm I\}$, we now have
$$Z_{\SU}(g)=\{\exp(t\xi)\mid \exp(\xi)=g \mbox{\ and\ }t\in\RR\}.$$

There are four obvious inclusions $\SU\to\Trefor$ corresponding to the
subsets of~$\Trefor$ where one of the entries is fixed at $\pm I$.
Since $H^3(\Trefor)$ has rank only~$2$, two pairs of these inclusions
must
be homotopic.
An explicit homotopy is as follows.

Let $\calS_1^+=\{(g,I)\mid g\in \SU\}$ and
let $\calS_1^-=\{(g,-I)\mid g\in \SU\}$.
For each $g\ne\pm I$, its centralizer determines a one-parameter
subgroup of~$\SU$ described as above.
We use this to  define a homotopy.
Define 
$$J_1:\calS_1^+\times [0,1] \to \Trefor, \ \ \
J_1\bigl((g,I),t\bigr):=(g,\exp(\pi t\xi/|\xi|) ),$$
where
$ \xi=\exp^{-1}(g)\in\circaccent{\g}$ for $g \ne \pm I$, and extends
uniquely by continuity to $g=\pm I$.
Then $J_1$ is a homotopy between the inclusions $\SU\to\Trefor$
corresponding to $\calS_1^+$ and~$\calS_1^- $.
Similarly there is a homotopy~$J_2$ between the inclusions
corresponding to
$\calS_2^+=\{(I,g)\mid g\in \SU\}$ and
$\calS_2^-=\{(-I,g)\mid g\in \SU\}$.
Set 
$$
\overline{\calS} = \calS_1^+\cup\calS_2^+=\{(g,h)\in\Trefor\mid g=I \mbox{\ or\ }h=I\}.
$$
Then $\overline{\calS}\cong S^3\vee S^3$ and
$\Trefor/\overline{\calS}\simeq S^2\vee\Sigma^2\,\RPt$.

\section{Atiyah space}\label{Atiyah}

Let $\mu:\SU^4\to \SU$ denote the product of commutators map,
$\mu\bigl((x,y,x',y')\bigr):=[x,y][x',y']$.
Then $-I$ is a regular value of $\mu$ which is fixed by the conjugation
action.
Define the the $6$-manifold
$A:=\mu^{-1}(-I)/\SU$, where $\SU$ acts by conjugation.
We refer to $A$ as the ``Atiyah space'' since it was studied by
Atiyah and Bott~\cite{AB}
who computed its cohomology groups using Morse theory.
The Atiyah space corresponds to the moduli space of flat connections on
a punctured $2$-hole torus with holonomy~$-I$ on the boundary circle
of the puncture, as explained in the introduction.
A complex algebraic variety homeomorphic to~$A$ was studied by Newstead
\cite{Newstead}.

\subsection{Retractions on $A$}
{\ }

\medskip
If $X$ represents an element of~$A$, then $X=(x,y,x',y')$
for some $x$, $y$, $x'$, $y'$ with $[x,y][x',y']=-I$.
There exists $\theta\in[0,\pi]$ such that $(x,y)\in W_\theta$,
and the value of $\theta$ is independent of the choice of
representative~$X$.
Moreover, one can choose the representative such that
$$(x,y)\in X_\theta.$$
With this choice of representative the condition
$[x,y][x',y']=-I$ implies that $(x',y')\in X_{\theta'}$
where $\theta':=\pi-\theta$.
For $\theta\in [0,\pi]$ set
$$A_\theta:=\{X\in A\mid (x,y)\in W_\theta
\mbox{\ for any representative $(x,y,x',y')$ of~$X$}\}.$$
For $S\subset [0,\pi]$, set
$A_S:=\cup_{\theta\in S} A_\theta$.
Notice that elements of $A_\theta$ have representatives
in $X_\theta\times X_{\theta'}$.

\begin{lemma}\label{Atheta}
{~}
\begin{itemize}
\item[a)] For $\theta\in(0,\pi)$,
$$A_\theta
\cong (X_\theta\times X_{\theta'})/T\cong \RP \times (\RP/T)
\cong \RP\times S^2$$
\item[b)] $A_0 \cong (X_0\times X_\pi)/ \SU \cong X_0$
\item[c)] $A_\pi \cong (X_\pi\times X_0)/ \SU \cong X_0$
\end{itemize}
\end{lemma}

\begin{proof}
{~}
\begin{itemize}
\item[a)]
The first homeomorphism is clear, since for $X\in A_{(0,\pi)}$ any two
representatives
in $X_\theta\times X_{\theta'}$ differ by the action of~$T$.
For the second homeomorphism, after applying the $T$-homeomorphisms
$X_\theta\cong \RP$, $X_{\theta'}\cong \RP$ we have
$(\RP\times \RP)/T$ where $T$ acts by left translation.
Consider the fibration
$$\RP\to (\RP\times \RP)/T \to \RP/T$$
This fibration has a retraction given by $(g,h)\mapsto g^{-1}h$
(which is well defined).

Thus $A_\theta\cong \RP\times \RP/T \cong \RP\times S^2$.
\item[b)]
In this case, the stabilizer is all of $\SU$ so
$$A_0 \cong (X_0\times X_\pi)/\SU. $$
Since $X_\pi/\SU=\pt$ the fibration
$$X_0\to (X_0\times X_\pi)/\SU\to X_\pi/\SU$$
gives $A_0\cong X_0$.

Similarly we get part~(c).
\end{itemize}
\end{proof}

Notice that the second  homeomorphism in~(a) is non-canonical --- we
could
equally well have reversed the roles of $\theta$ and~$\theta'$.
This is reflected in a lack of symmetry in the Mayer-Vietoris
sequence below for~$H^*(A)$.

The existence of deformation retractions
$\Trefor\cong A_0\simeq A_{[0,\pi)}$ and
$\Trefor\cong A_\pi\simeq A_{(0,\pi]}$
follows from Theorem~\ref{MilnorMorse}.
Furthermore, the methods of subsection~\ref{gradientflow} show that such
retractions can be obtained via gradient flow.

In order to compute transition functions we will use another deformation
retraction making use of the maps in
subsection~\ref{explicitretraction}.
Define a deformation  as follows.
For $\theta\in(0,\pi)$ let $(x,y,x',y')\in X_\theta\times X_{\theta'}$
represent an element of~$A_\theta$.
Define
$\rho:A_{(0,\pi)}\times [0,1] \to A_{[0,\pi)}$
with $\im\rho_0\subset A_0$ by
\begin{equation}\label{definitionrho}
\rho(\bfl(x,y,x',y')\bfr,t):=
\leftbfl\Bigl(r\bigl((x,y),t\bigr),r'\bigl((x',y'),t\bigr)\Bigr)\rightbfr,
\end{equation}
where $r$ and $r'$ were defined in equations (\ref{definitionr}),
(\ref{definitionextendr}) and (\ref{definitionr'}).

Notice that $(x,y,x',y')\in X_\theta\times X_{\theta'}$ implies
\[
\mu\Bigl(r\bigl((x,y),t\bigr),r'\bigl((x',y'),t\bigr)\Bigr)
=e^{it\theta}e^{i(t(\pi-\theta)+(1-t)\pi)}=e^{i\pi}=-I.
\]
That is,
$\rho\left(\bfl(x,y,x',y')\bfr,t\right)\in A_{t\theta}$.

If $u(x,y,x',y')u^{-1}$ is another representative
for $\bfl(x,y,x',y')\bfr$ in
$A_\theta \cong (X_\theta\times X_{\pi-\theta})/T$,
where $u\in T$,
then $$r\bigl(u(x,y)u^{-1},t\bigr)=ur(x,y,t)u^{-1}$$
since $r$ is $T$-equivariant, and similarly
$$r'\bigl(u(x',y')u^{-1},t\bigr)=ur'(x',y',t)u^{-1}. $$
Thus $\rho$ is well defined.

Extend $\rho$ to $A_0\times [0,1]$ by projection onto the first
factor to get a deformation
\begin{equation}\label{definitionextendrho}
\rho:A_{[0,\pi)}\times [0,1]\to A_{[0,\pi)}.
\end{equation}
Similarly we have a deformation
$\rho':A_{(0,\pi]}\times [0,1]\to A_{(0,\pi]}$
with $\im\rho'_0\subset A_\pi.$

\subsection{Cohomology of $A$}\label{sec:cohomologya}
{\ }
\medskip

It turns out that $A$ is simply connected with $H^*(A)$ torsion-free and
its Betti numbers
were computed by Atiyah and Bott~\cite{AB} via Morse theory.
In this section we will do the computation using the Mayer-Vietoris sequence.

As above, write $A=U\cup V$ where we set $U:=A_{[0,\pi)}\simeq A_0\simeq
\Trefor$;
$V:=A_{(0,\pi]}\simeq A_\pi\simeq\Trefor$. Since $U$, $V$ are simply
connected and $U\cap V$ is connected, it follows from Van Kampen's
theorem that $A$ is simply connected.

Set
$$\calS = \{\bfl(x,y,x',y')\bfr\in A \mid
\mbox{\ at least one of $x$, $y$, is\ $I$}\}
\cong S^3\wwedge S^3$$
and
$$\calS' = \{\bfl(x,y,x',y')\bfr\in A \mid
\mbox{\ at least one of $x'$, $y'$, is\ $I$}\}
\cong S^3\wwedge S^3.$$
Set $\vv{U}:=U/\calS$, $\vv{V}:=V/\calS'$.
Set $\vv{A}:=\vv{U}\cup\vv{V}$.
Then $\vv{A}=A/\mathord{\sim}$ where
$\bfl(I,y_1,x',y')\bfr\sim\bfl(I,y_2,x',y')\bfr. $
Similarly, with $I$ in any other position, the partner of 
$   I$ can be replaced by any element of $\SU$.

In view of Lemma~\ref{Atheta}, this can be expressed as
$\vv{A}=A/\mathord{\sim}$ where
$$\bfl(I,y,x',y')\bfr\sim\bfl(x,I,x',y')\bfr
\sim\bfl(I,I,\pt)\bfr$$ and
$$\bfl(x,y,I,y')\bfr\sim\bfl(x,y,x',I)\bfr
\sim\bfl(\pt,I,I)\bfr,$$ where $*\cong X_\pi/\SU$.

Note that the composition $\calS\to U\simeq\Trefor$ induces
an isomorphism on~$H^3(~)$ and similarly $\calS'\to V\simeq\Trefor$
induces an isomorphism on~$H^3(~)$.
Therefore $\tilde{H}^*(U)\cong \tilde{H}^*(\vv{U})\oplus
\tilde{H}^*(\calS)$
where $\tilde{H}^q(\calS)=\ZZ\oplus\ZZ$ if $q=3$ and $0$ otherwise.
and similarly
$\tilde{H}^*(V)\cong \tilde{H}^*(\vv{V})\oplus \tilde{H}^*(\calS')$.

We calculate the groups $H^*(\vv{A})$ by means of the Mayer-Vietoris
sequence for $\vv{A}=\vv{U}\cup\vv{V}$ and then obtain the cohomology
of~$A$
from $H^*(A)\cong H^*(\vv{A})\oplus \tilde{H}^*(\calS)\oplus
\tilde{H}^*(\calS')$.
Set 
$$B:=\vv{U}\cap\vv{V} = U\cap V=A_{(0,\pi)}\cong
\RP \times S^2\times(0,\pi).$$

First we calculate with $\FF$~coefficients. Recall that $\FF$ denotes
the finite field with two elements.
With these coefficients
$$H^*(\vv{U})\cong H^*(\vv{\Trefor})= \langle 1, a, b, c\rangle$$ where
$|a|=2$,
$|b|=3$, $|c|=4$,
and
$$\vv{\Trefor} = \Trefor/\{ (x,y)\in \Trefor \mid x = I \mbox{ \ or \ }
y=I\}. $$
Similarly
$H^*(\vv{V})=\langle 1, a',b', c'\rangle$ where $|a'|=2$, $|b'|=3$,
$|c'|=4$.

Also $$H^*(B)=\langle 1, s, t, t^2, st, t^3 , st^2, st^3 \rangle$$
where $|s|=2$, $|t|=1$.

The Bockstein is given by $\beta(b)=c$, $\beta(b')=c'$, $\beta(t)=t^2$,
$\beta(a)=\beta(a')=\beta(s)=0$.

We have
$$\displaylines{
0 \rTo \underbrace{H^1(B)}_\FF \rTo
H^2 (\vv{A}) \rTo \underbrace{H^2 (\vv{U})}_\FF \oplus
\underbrace{H^2 (\vv{V})}_\FF
\rTo \underbrace{H^2 (B)}_{\FF+\FF}\rTo\cr
H^3 (\vv{A}) \rTo \underbrace{H^3 (\vv{U})}_\FF \oplus
\underbrace{H^3 (\vv{V})}_\FF
\rTo \underbrace{H^3 (B)}_{\FF+\FF}\rTo\cr
H^4 (\vv{A}) \rTo \underbrace{H^4 (\vv{U})}_\FF \oplus
\underbrace{H^4 (\vv{V})}_\FF
\rTo \underbrace{H^4 (B)}_\FF\rTo\cr
H^5 (\vv{A}) \rTo \underbrace{H^5 (\vv{U})}_0 \oplus
\underbrace{H^5 (\vv{V})}_0
\rTo \underbrace{H^5 (B)}_\FF\rTo\cr
H^6 (\vv{A}) \rTo 0.
}$$

Let $j $ denote the inclusion $ B \to\vv{U}$,
and let $j'$ denote the inclusion $B \to\vv{V}$.
The group
$$H^2(B)\cong H^2(S^2\times\RP)\cong H^2(S^2)\oplus H^2(\RP), $$
where the first isomorphism involves a choice.
\smallskip

As in the computation of $\Trefor$, the map
$S^2\to A_0\cong \Trefor\to \vv{\Trefor}=\vv{U}$ has degree one on
integral
cohomology regardless of the choice.
By symmetry, the map $S^2\to A_{\pi}\to\vv{V}$ has degree one.
However the maps $\RP\to A_0\to\vv{U}$ and $\RP\to A_\pi\to\vv{V}$
depend on
the choice.
If we make the choice in which $\RP\to\vv{U}$ is null homotopic, then
$\RP\to\vv{V}$ will induce an isomorphism on~$H^2(~;\FF)$.
With this choice, we have
$j^*(a) = s$ and $(j')^*(a') = s + t^2$
and we see that $j^*\bot (j')^*$ is an isomorphism on~$H^2(~;\FF)$.
\smallskip

With the above choice of $B\simeq S^2\times \RP$ we have
$j^*(b) = st$;
$(j')^*(b') = t^3 +st$
and so $j^*\bot (j')^*$ is an isomorphism on~$H^3(~;\FF)$.
\smallskip
Applying the Bockstein gives
$j^*(c) =   st^2$;
$(j')^*(c') =  st^2$.
In degree 4 the kernel is $c + c'$. The cokernel is $0$.
\smallskip
In degree 5, $j^*$ and $(j')^*$ are zero maps and the cokernel is
$st^3$.
It follows:
\begin{theorem}\label{thm:cohomologya}
$A$ is simply connected.
The mod~$2$ cohomology of~$A$ is given by
$$H^q(A;\FF)=
\begin{cases}
\FF&q=0,2,4,6;\cr
\FF^4&q=3;\cr
0&\mbox{otherwise}.\cr
\end{cases}
$$
where
\begin{itemize}
\item[] $H^2(A;\FF)$ is generated by $\delta(t)$;
\item[] $H^3(A;\FF)$ is generated by $s_1$, $s_2$, $s_1'$, $s_2'$;
\item[] the generator of $H^4(A;\FF)$ maps to $c+c'$;
\item[] $H^5(A;\FF)$ is generated by $\delta(st^3)$.
\end{itemize}
Here $\delta$ is the connecting homomorphism.
With integer coefficients this gives
$$H^q(A)=
\begin{cases}
\ZZ&q=0,2,4,6;\cr
\ZZ^4&q=3;\cr
0&\mbox{otherwise.}\cr
\end{cases}
$$
\end{theorem}
This reproduces the result of Atiyah and Bott~\cite{AB}.
We will obtain the ring structure of $H^*(A)$ in the next section.

Note that although $\vv{A}$ is not manifestly homeomorphic to a manifold,  we
will see in section~\ref{wall} that $A$ is diffeomorphic to a connected sum
$$A\cong A'\#(S^3\times S^3)\#(S^3\times S^3)$$
where $A'$ is homeomorphic to~$\vv{A}$.
However we have neither reproved nor used this fact.

\section{Prequantum line bundle over~$A$}\label{qAtiyah}

The manifold~$A$ has a symplectic structure.
It is an example of the reduction of a quasi-Hamiltonian
$G$-space~\cite{AMM98}.
Denote by $\omega$ the resulting symplectic form on $A$.  
Line bundles on $A$ are represented by elements of $H^2(~)$. The prequantum
line bundle is represented by the cohomology class of~$\omega$.

Let $1$, $x$, $s_1$, $s_2$, $s_3$, $s_4$, $y$, $z$ denote group
generators
of $H^*(A)$ in degrees $0$, $2$, $3$, $3$, $3$, $3$, $4$, $6$,
respectively.
We may choose $x\in H^2(A)$ to be the cohomology class represented
by~$\omega$.
Abusing notation, we write $H^*(\vv{A})=\langle1,x,y,z\rangle$
and for their images
under $H^*(\vv{A})\to H^*(A)$.
Let $L$ denote the line bundle over~$\vv{A}$ associated to~$x$.
Thus the pullback of $L$ to~$A$ is the prequantum line bundle $L_A$
of~$A$.

Let $\qb{A}$ denote the total space of the $S^1$-bundle $P$ over
$\vv{A}$ associated with~$L$.
Let $\vv{U}$, $\vv{V}$, $B$ be as in \S~\ref{Atiyah}.
The restriction of~$P$  gives $S^1$-bundles over those subspaces.
Denote the total spaces by $\qb{U}$, $\qb{V}$, $\qb{B}$ respectively.
The Mayer-Vietoris sequence for $\vv{A}$ tells us
that $H^2(\vv{A})\to H^2(\vv{U})$ and $H^2(\vv{A})\to H^2(\vv{V})$ have
degree~$2$.
Thus the restrictions of $P$ to $U$,$V$, and~$B$ are classified
by $2a$, $2a'$, and~$2s$, using the notation of \S\ref{Atiyah}.

The total space~$\qb{B}$ of the $S^1$-bundle over $B\cong S^2\times
\RP\times (0,\pi)$
is homotopy equivalent to $\RP\times\RP$.
In fact we initially formed~$B$ in Lemma~\ref{Atheta}
as~$((\RP\times\RP)/T)\times (0,\pi)$.
Therefore,
we have
$$\tilde{H}^*(\qb{B};\FF)=\langle s,t,s^2,t^2,st,s^3,t^3,s^2t,st^2,
s^3t,st^3,s^2t^2,s^3t^2,s^2t^3,s^3t^3\rangle$$
with $s,t\in H^1(\RP;\FF)$.

We compute first using $\FF$~ coefficients.
Let $v$ denote the preimage in $H^*(\qb{U})$ of the generator
of~~$H^1(S^1)$.
From the Serre spectral sequence we get
$\tilde{H}^*(\qb{U})=\langle v, a, va, b, vb, c, vc\rangle$
in degrees $1, 2, 3, 3, 4, 4, 5$ respectively.
The Bockstein is determined by $\beta(v)=a$, $\beta(b)=c$
(which implies $a=v^2$ although we will not need this fact later).
Similarly we have
$\tilde{H}^*(\qb{V})=\langle v', a', v'a', b', v'b', c', v'c'\rangle$.
Taking inverse images of $\vv{A}=\vv{U}\cup_B\vv{V}$ under the bundle
projection gives $\qb{A}=\qb{U}\cup_{\qb{B}}\qb{V}$.
The associated Mayer-Vietoris sequence with $\FF$ coefficients is

$$\displaylines{
H^1 (\qb{A}) \rTo
\underbrace{H^1 (\qb{U})}_\FF \oplus \underbrace{H^1 (\qb{V})}_\FF
\rTo \underbrace{H^1 (\qb{B})}_{\FF+\FF}\rTo \cr
H^2 (\qb{A}) \rTo
\underbrace{H^2(\qb{U})}_{\FF} \oplus
\underbrace{H^2 (\qb{V})}_{\FF} \rTo
\underbrace{H^2 (\qb{B})}_{\FF^3}\rTo \cr
H^3 (\qb{A}) \rTo
\underbrace{H^3 (\qb{U})}_{\FF+\FF} \oplus
\underbrace{H^3 (\qb{V})}_{\FF+\FF}
\rTo\underbrace{H^3 (\qb{B})}_{\FF^4}\rTo \cr
H^4 (\qb{A}) \rTo
\underbrace{H^4 (\qb{U})}_{\FF+\FF} \oplus
\underbrace{H^4 (\qb{V})}_{\FF+\FF}
\rTo \underbrace{H^4 (\qb{B})}_{\FF^3}\rTo \cr
H^5 (\qb{A}) \rTo
\underbrace{H^5 (\qb{U})}_{\FF} \oplus
\underbrace{H^5 (\qb{V})}_{\FF}
\rTo\underbrace{H^5 (\qb{B})}_{\FF+\FF}\rTo \cr
H^6 (\qb{A}) \rTo
\underbrace{H^6 (\qb{U})}_0 \oplus
\underbrace{H^6 (\qb{V})}_0
\rTo \underbrace{H^6 (\qb{B})}_{\FF}\rTo \cr
H^7 (\qb{A}) \rTo 0. \cr
}$$
The maps $H^*(\qb{U})\to H^*(\qb{B})$ and~$H^*(\qb{V})\to H^*(\qb{B})$
are
determined by the corresponding Mayer-Vietoris sequence for~$\vv{A}$ and
we get
\begin{theorem} \label{cohomologyM'}
The mod $2$ cohomology of $\qb{A}$ is given by
$$H^q(\qb{A};\FF)=
\begin{cases}
\FF&q=0,3,4,7;\cr
0&\mbox{otherwise}.\cr
\end{cases}
$$
\end{theorem}

Now consider the sequence with integer coefficients.
$$\displaylines{
H^1 (\qb{A}) \rTo
\underbrace{H^1 (\qb{U})}_0 \oplus \underbrace{H^1 (\qb{V})}_0
\rTo \underbrace{H^1 (\qb{B})}_{0}\rTo \cr
H^2 (\qb{A}) \rTo
\underbrace{H^2(\qb{U})}_{\ZZ/2} \oplus
\underbrace{H^2 (\qb{V})}_{\ZZ/2} \rTo
\underbrace{H^2 (\qb{B})}_{\ZZ/2+\ZZ/2}\rTo \cr
H^3 (\qb{A}) \rTo
\underbrace{H^3 (\qb{U})}_{\ZZ} \oplus
\underbrace{H^3 (\qb{V})}_{\ZZ}
\rTo\underbrace{H^3 (\qb{B})}_{\ZZ+\ZZ+\ZZ/2}\rTo \cr
H^4 (\qb{A}) \rTo
\underbrace{H^4 (\qb{U})}_{\ZZ/2} \oplus
\underbrace{H^4 (\qb{V})}_{\ZZ/2}
\rTo \underbrace{H^4 (\qb{B})}_{\ZZ/2}\rTo \cr
H^5 (\qb{A}) \rTo
\underbrace{H^5 (\qb{U})}_{\ZZ/2} \oplus
\underbrace{H^5 (\qb{V})}_{\ZZ/2}
\rTo\underbrace{H^5 (\qb{B})}_{\ZZ/2+\ZZ/2}\rTo \cr
H^6 (\qb{A}) \rTo
\underbrace{H^6 (\qb{U})}_0 \oplus
\underbrace{H^6 (\qb{V})}_0
\rTo \underbrace{H^6 (\qb{B})}_{\ZZ}\rTo \cr
H^7 (\qb{A}) \rTo 0. \cr}$$

The segment
$0\to \coker\delta\to H^4(\qb{A})\to \ker\delta\to0 $
becomes $$0\to \ZZ/2\to H^4(\qb{A})\to \ZZ/2\to0,$$
so we see that $H^4(\qb{A})$ has $4$~elements.
Our calculation with $\FF$~coefficients shows that the
mod~$2$ reduction $H^4(\qb{A};\FF)$ of $H^4(\qb{A})$ has
a single summand.
Thus $H^4(\qb{A})=\ZZ/4$.

Therefore the integral cohomology of $\qb{A}$ is
\begin{theorem} \label{zcohqA'}
$$H^q(\qb{A})=
\begin{cases}
\ZZ&q=0,7;\cr
\ZZ/4&q=4;\cr
0&\mbox{otherwise}.\cr
\end{cases}
$$
\end{theorem}

We can now obtain the ring structure in the integral cohomology
$H^*(A)$.
Let $1$, $x$, $s_1$, $s_2$, $s_3$, $s_4$ $y$, $z$ denote group
generators
of $H^*(A)$ in degrees $0$, $2$, $3$, $3$, $3$, $3$, $4$, $6$,
respectively.
The elements $s_1$, $s_2$, $s_3$, $s_4$ are torsion-free elements of
odd degree so their squares are $0$.
By Poincar\'e duality, $s_1 s_2=s_3 s_4=xy=z$ with appropriate choices
of signs of the generators.
Write $x^2=\lambda y$, where we may choose the signs of the
generators so that $\lambda\ge0$.

Since we calculated above that $H^1(\qb{A})=0$,
in the cohomology Serre spectral sequence for the principal fibration
$S^1\to \qb{A}\to \vv{A} $ we must have $d(v)=x$, where $v$ is a
generator
of~$H^1(S^1)$.
Therefore $d(vx)=x^2=\lambda y$.
Thus the spectral sequence gives~$H^4(\qb{A}) =\ZZ/\lambda $.
Comparing this with our calculation above, we conclude that $\lambda=4$.
Thus we have
\begin{theorem} \label{ringstructure} 
As a ring,
$H^*(A;\ZZ)=\langle x,s_1,s_2,s_3,s_4, y,z\rangle $
with the nontrivial products given  by $s_1s_2=s_3s_4=xy=z$, $x^2=4y$.
\end{theorem}
This agrees with the following result given by Thaddeus \cite{Thaddeus}
using methods from
algebraic geometry.
More generally, Thaddeus shows
$$x^m =  (-1)^g 2^{2g-2} \frac{m!}{(m-g+1)!}(2^{m-g+1} -2) B_{m-g+1} z
$$
where $B_k$ is the $k${\it th} Bernoulli number
and in our case, $m =3$ and $g =2$ and $z$ is the volume form.
This works out to
$x^3 = 4 z $ as above.

\section{The $9$-manifold $M:=\mu^{-1}(-I)$}\label{nineman}

As noted earlier, $-I$ is a regular value of the product of commutators
map~$\mu$, so $M:=\mu^{-1}(-I)$ is a $9$-manifold and $A:=M/\SU$
where $\SU$ acts diagonally by conjugation.
The stabilizer of the $\SU$ action in $A$ is the center of $\SU$,
so equivalently we have a free action of $\SO\cong \RP$ on~$M$, with
$A\cong M/\RP$.

Let $q:M\to A$ be the $\RP$-bundle projection.
For $\theta\in[0,\pi]$, set
$$M_\theta:=q^{-1}(A_\theta)=\{(x,y,x',y')\in M \mid
[x,y]\sim e^{i\theta}\mbox{\ and\ }[x,'y']\sim-e^{-i\theta}\}$$

The deformation retraction $\rho:A_{[0,\pi)}\times [0,1] \to A_{[0,\pi)}$ defined
in equation (\ref{definitionrho}) and (\ref{definitionextendrho})
is covered by a deformation retraction
$\hat\rho:M_{[0,\pi)}\times [0,1] \to M_{[0,\pi)}$ with
$\im \hat\rho_0\subset M_0$
given by
\begin{equation}\label{definitionrhohat}
\hat\rho\bigl ( (x,y,x',y'),t\bigr):=\bigl ( r_W(x,y,t), r'_W(x',y',t)
\bigr),
\end{equation}
where $r_W$ and $r'_W$ were defined in equation (\ref{definitionrw}).
Similarly $\rho':A_{(0,\pi]}\times [0,1] \to A_{(0,\pi]}$
is covered by a deformation
${\hat\rho}':M_{(0,\pi]}\times [0,1] \to M_{(0,\pi]}$
with $\im {\hat\rho}'_0\subset M_\pi$.

As in Section~\ref{Atiyah}, set $U:=A_{[0,\pi)}$, $V:=A_{(0,\pi]}$.
Set $\hat{U}:=q^{-1}(U)=M_{[0,\pi)}$; $\hat{V}:=q^{-1}(V)=M_{(0,\pi]}$;
$\hat{\calS}:=q^{-1}(\calS)$; $\hat{\calS'}:=q^{-1}(\calS')$;
$\hat{\vv{U}}=\hat{U}/\hat{\calS}$;
$\hat{\vv{V}}=\hat{V}/\hat{\calS'}$;
$\hat{B}=\hat{\vv{U}}\cap\hat{\vv{V}}=
q^{-1}(U\cap V)$; $\vv{M}=\hat{\vv{U}}\cup \hat{\vv{V}}=q^{-1}(\vv{A})$.

\subsection{Local Trivialization of $M\to A$}

\begin{lemma}
The restrictions of the principal bundle $\RP\to M\to A$ to $U$ and
to~$V$
are trivial.
\end{lemma}

\begin{proof}
Since we have retractions $U\to X_0$ and $V\to X_{\pi}$, it suffices to
show
that the restrictions of the bundle to $X_0$ and $X_\pi$ are trivial.
Consider
$$M_0 = \{(x,y,x',y')\in \SU^4 \mid [x,y]=I\mbox{\ and\ }[x',y']=-I\}
=X_0\times X_\pi.$$
Under the conjugation action, $X_\pi/\SU$ is isomorphic
to $\RP/\SU$ with translation action, which is a point.
So the restriction of the bundle to $X_0$ is trivial.
Similarly $M_\pi = X_\pi\times X_0$,
and under diagonal conjugation action the first factor becomes a point,
so the restriction to $X_\pi$ is trivial.
\end{proof}

Let $\Delta:M_0\to \RP$ be the composite
$$M_0 \cong X_0\times X_\pi \rTo^{\pi_2} X_\pi
\rTo^{\Phi_\pi}_{\cong}\RP,$$
where $\Phi_\theta$ is the homeomorphism defined in
Theorem~\ref{thm:explicithom}.
A trivialization of $q:\hat{U}\to U$ is given by
$$X\mapsto \bigl(\Delta\circ \hat{\rho_0}(X) ,q(X)\bigr)
\in X_\pi\times U\cong \RP\times U,\ X=(x,y,x',y')\in
\hat{U}:=M_{[0,\pi)},$$
where $\hat{\rho_0}$ is given by equation (\ref{definitionrhohat}).
Similarly a trivialization of $q:\hat{V}\to V$ is given by
$$X\mapsto \bigl(\Delta'\circ \hat{\rho}'_0(X) ,q(X)\bigr)
\in X_\pi\times V\cong \RP\times V,\ X=(x,y,x',y')\in
\hat{V}:=M_{(0,\pi]},$$
where $\Delta':M_\pi\to\RP$ is the composition
$$M_\pi \cong X_\pi\times X_0
\rTo^{\pi_1} X_\pi \rTo^{\Phi_\pi}_\cong\RP.$$

Self maps of trivial bundles are in one to one correspondence with
homotopy classes of maps from the base to the group.
For $\theta\in(0,\pi)$, the restriction of the transition function
$\tau:U\cap V\to \RP$ to $A_{\theta}\subset U\cap V$ is given by
$$\tau\left(\bfl(x,y,x', y'\bfr\right)=
\Bigl(\Phi_\pi\bigl(r'(x',y',0)\bigr)\Bigr)^{-1}
 \Phi_\pi\bigl(r'(x,y,0)\bigr),
\quad \bfl(x,y,x', y'\bfr\in A_\theta \cong (X_\theta\times
X_{\pi-\theta})/T,$$
where $r'$ is defined in equation (\ref{definitionr'}).
Note that the function $\tau$ is well defined because $r'$ and
$\Phi_\pi$ are $T$-maps, and $\im(r'_0)\subset X_\pi$.

\begin{theorem}
For $\theta\in (0,\pi)$, the restriction $\tau_\theta:A_\theta\to\RP$ of
the transition function to $A_\theta$ is homotopic to the
function
$\bfl(x,y,x',y')\bfr\mapsto
(\Phi_{\pi-\theta}(x',y')\bigr)^{-1}\Phi_\theta(x,y)$
for $\bfl(x,y,x',y')\bfr\in A_\theta.$
In particular, the transition function $\tau:U\cap V\to \RP$ is
homotopic to the composite $U\cap V\simeq (\RP\times \RP)/T \to \RP$
where the final map is given by $(g,h)\mapsto g^{-1}h$.
\end{theorem}

\begin{proof}
The restriction $\tau_\theta$ of the transition function is defined by
$$\tau\left(\bfl(x,y,x',y')\bfr\right) =
(\Phi_\pi\bigl(r'(x',y',0)\bigr)^{-1}\Phi_\pi\bigl(r'(x,y,0)\bigr),
\quad\forall~ \bfl(x,y,x',y')\bfr\in A_\theta.$$
Define a homotopy $J:A_\theta \times [0,1] \to \RP$ by
$$J\left(\bfl(x,y,x',y')\bfr,t\right) :=
(\Phi_{t(\pi-\theta)+(1-t)\pi}\bigl(r'(x',y',t)\bigr)^{-1}
\Phi_{t\theta+(1-t)\pi}\bigl(r'(x,y,t)\bigr).$$
Then
$$J\left(\bfl(x,y,x',y')\bfr,1\right) =
(\Phi_{\pi-\theta}(x',y')\bigr)^{-1}\Phi_\theta(x,y).$$

\end{proof}

\begin{corollary}
On cohomology with $\FF$ coefficients, the transition function satisfies
$\tau^*(v)=t$ where $v$ is the generator of $H^1(\RP;\FF)$ and $t$ is
the
generator of $H^1(B;\FF)$.
\end{corollary}

\begin{proof}
The preceding theorem shows that $\tau$ corresponds under the homotopy
equivalence
$B\simeq(\RP\times\RP)/T$ to the map $(g,h)\mapsto g^{-1}h$.
Thus under the composition
$$\RP\rTo^{\iota_2} S^2\times \RP\times (0,\pi) \cong B\simeq (\RP\times
\RP)/T\rTo^{\tau}\RP$$
$g$ drops out and we get the identity.
Hence $\tau^*(v)=t$.
\end{proof}

\subsection{Cohomology of $\vv{M}$} \label{ss:cohvm}
{\ }
\smallskip

Set $\hat{B}:=q^{-1}(B)\cong \RP\times B$.
Then $\vv{M} = \hat{\vv{U}} \cup_{\hat{B}}  \hat{\vv{V}} $
where the bundle is trivial over $U$ and $V$.
As before, let
$$\vv{\Trefor} = \Trefor/\{ (x,y)\in \Trefor \mid x = I \mbox{ \ or \ }
y=I \}. $$
Then
$$ \tilde{H}^*(\vv{\Trefor}) = \langle a,b,c\rangle $$
with $|a| = 2$, $|b| = 3$, $|c| = 4 $, $\beta(b)=c$.

First we consider coefficients in $\FF$.  The
K\"unneth formula gives
$$H^*(\hat{\vv{U}}) \cong H^*(\RP) \otimes H^*(\vv{U})\cong
H^*(\RP)\otimes
H^*(\vv{\Trefor}).$$
Similarly
$$H^*(\hat{\vv{V}})\cong H^*(\RP)\otimes H^*(\vv{V})\cong
H^*(\RP)\otimes
H^*(\vv{\Trefor}).$$
Also
$$H^*(\hat{B})\cong
H^*(\RP)\otimes H^*(B)\cong H^*(\RP)\otimes H^*(S^2)\otimes H^*(\RP),$$
where the isomorphism is not canonical.

With $\FF$ coefficients,
$\tilde{H}^*(\hat{\vv{U}})$ is generated by $a, b, c$  from the base and
$w, w^2, w^3$ from the fiber.
For $\tilde{H}^*(\hat{\vv{V}})$ we denote the generators by
$a'$, $b'$, $c'$, $w'$, $(w')^2$ and~$(w')^3$.
For $H^*(\hat{B})$, the ring generators are $s$ (in degree 2) and $t$ in
degree 1 (from the base)
and $v$ in degree 1 (from the fiber).
As a ring
$$H^*(\hat{B}) = \langle s, t, v\rangle $$
where $|s| = 2$, $|t| = 1 $, $|v|=1$,
$s^2 = 0$; $t^4 = 0$; $v^4=0$.

The transition function $\tau$ is the map ${B} \to \RP$
described in the previous subsection.
Associated to a transition function is a self-homeomorphism
$\tilde{\tau} : B\times \RP \cong B\times \RP$ given by
$(b,g)\mapsto \bigl(b,\tau(b) g\bigr)$.

\begin{lemma}
On mod~$2$ cohomology
$\tilde{\tau}^*(v)=v+t$
\end{lemma}

\begin{proof}
$v\in H^1(\hat{B})\cong H^1(B\times\RP)$ is defined as a preimage
of the generator of $H^1(\RP)$ from the fibre.
Equivalently, having chosen a trivialization $\hat{B}\cong B\times\RP$,
$v$ is the image of the generator of $H^1(\RP)$ under
$\pi_2:\hat{B}\to\RP$.
The composition $B\times\RP=\hat{B}\rTo^{\bar{\tau}} \hat{B}\rTo^{\pi_2}
\RP $
is given by $(b,g)\mapsto \tau(b)g$ so
$\tilde{\tau}^*(v)=(1_\RP)^*(v)+\tau^*(v)=v+t$.
\end{proof}

Let $j:B\to U$ and $j':B\to V$ denote the inclusions.
Recall from the Mayer-Vietoris sequence for $A$ that
based on our choice of isomorphism:
$$j^*(a)= s; \ \ j^*(b)= st; \ \ j^*(c)= st^2; \ \ 
(j')^*(a')= s+t^2; \ \ (j')^*(b')= st+t^3; \ \ (j')^*(c')= st^2.$$
Below, $f$ denotes $j\times 1_\RP$ while
$f'$ denotes $(j'\times 1_\RP)\circ\tilde{\tau}$.
$$\displaylines{
H^1 (\vv{M}) \rTo
\underbrace{H^1 (\hat{\vv{U}})}_\FF \oplus \underbrace{H^1
(\hat{\vv{V}})}_\FF
\rTo^{\bigl(f \bot f'\bigr)^*}
 \underbrace{H^1 (\hat{B})}_{\FF+\FF}\rTo \cr
H^2 (\vv{M}) \rTo
\underbrace{H^2(\hat{\vv{U}})}_{\FF+\FF} \oplus
\underbrace{H^2 (\hat{\vv{V}})}_{\FF+\FF} \rTo^{\bigl(f \bot f'\bigr)^*}
\underbrace{H^2 (\hat{B})}_{\FF^4}\rTo \cr
H^3 (\vv{M}) \rTo
\underbrace{H^3 (\hat{\vv{U}})}_{\FF^3} \oplus
\underbrace{H^3 (\hat{\vv{V}})}_{\FF^3}
\rTo^{\bigl(f \bot f'\bigr)^*} \underbrace{H^3 (\hat{B})}_{\FF^6}\rTo
\cr
H^4 (\vv{M}) \rTo
\underbrace{H^4 (\hat{\vv{U}})}_{\FF^3} \oplus
\underbrace{H^4 (\hat{\vv{V}})}_{\FF^3}
\rTo^{\bigl(f \bot f'\bigr)^*} \underbrace{H^4 (\hat{B})}_{\FF^6}\rTo
\cr
H^5 (\vv{M}) \rTo
\underbrace{H^5 (\hat{\vv{U}})}_{\FF^3} \oplus
\underbrace{H^5 (\hat{\vv{V}})}_{\FF^3}
\rTo^{\bigl(f \bot f'\bigr)^*} \underbrace{H^5 (\hat{B})}_{\FF^6}\rTo
\cr
H^6 (\vv{M}) \rTo
\underbrace{H^6 (\hat{\vv{U}})}_{\FF+\FF} \oplus
\underbrace{H^6 (\hat{\vv{V}})}_{\FF+\FF}
\rTo^{\bigl(f \bot f'\bigr)^*} \underbrace{H^6 (\hat{B})}_{\FF^4}\rTo
\cr
H^7 (\vv{M}) \rTo
\underbrace{H^7 (\hat{\vv{U}})}_\FF \oplus\underbrace{ H^7
(\hat{\vv{V}})}_\FF
\rTo^{\bigl(f \bot f'\bigr)^*} \underbrace{H^7 (\hat{B})}_{\FF+\FF}\rTo
\cr
H^8 (\vv{M}) \rTo
\underbrace{H^8 (\hat{U'})}_0 \oplus \underbrace{H^8 (\hat{\vv{V}})}_0
\rTo^{\bigl(f \bot f'\bigr)^*} \underbrace{H^8 (\hat{B})}_\FF\rTo \cr
H^9 (\vv{M}) \rTo 0. \cr
}$$

The map
$\bigl((j\times 1_\RP) \bot(j'\times 1_\RP)\circ\tilde{\tau}\bigr)^*$
is
\begin{itemize}
\item[1)]
$w \mapsto v$
\item[]
$ w' \rTo^{j'\times 1_\RP} v\rTo^{\bar{\tau}} v + t $

To compute $\bigl((j'\times 1_\RP)\circ\tilde{\tau}\bigr)^* (w')$,
we used that
$$H^1(\hat{B})  \cong H^1(S^2 \times \RP \times \RP) \cong H^1( \RP
\times \RP),$$
while
$H^1(\hat{\vv{V}}) \cong  H^1(\vv{V} \times \RP) \cong  H^1(\RP) $
since $\vv{V}$ is
simply connected.
On $H^1(~)$, effectively the map is given by the multiplication map on
$\RP$,
and so on cohomology $w' \mapsto t + v$ using Corollary 6.1.

The above  is an isomorphism.

\item[2)]
$f^*: a \mapsto s$
\item[]
$f^*: w^2 \mapsto v^2 $
\item[]
$(f')^*: a' \mapsto s+t^2 $
\item[]
$(f')^*: (w')^2 \mapsto v^2 + t^2 $

$Ker =  < a + w^2 + a' + (w')^2 >$

$Coker =< vt >$

\item[3)]
$f^*: a w \mapsto sv$
\item[]

$f^*: b \mapsto st$
\item[]

$f^*: w^3 \mapsto v^3 $
\item[]

$(f')^*: a' w' \mapsto (s+t^2) (v+t) = sv+st+t^2v + t^3 $
\item[]

$(f')^*: b' \mapsto st+ t^3$
\item[]
$(f')^*: (w')^3 \mapsto (v+t)^3 = v^3 + v^2 t + vt^2 + t^3 $

This is an isomorphism.

\item[4)]
$f^*: aw^2 \mapsto sv^2$
\item[]

$f^*: bw \mapsto stv$
\item[]

$f^*: c \mapsto st^2 $
\item[]

$(f')^*: a' (w')^2 \mapsto (s+t^2) (v^2 + t^2) = sv^2+st^2 +t^2 v^2 $
\item[]

$(f')^*: b' w' \mapsto (st + t^3)(v+t) = stv+ st^2 + vt^3 $
\item[]

$(f')^*: c' \mapsto st^2 $

$Ker = < c+ c'>$

$Coker = <v^3 t>$

\item[5)]
$ f^*: a w^3 \mapsto s v^3 $
\item[]
$ f^*: bw^2 \mapsto stv^2 $
\item[]
$ f^*: cw \mapsto st^2 v $
\item[]
$ (f')^*: a' (w')^3  \mapsto (s+t^2) (v+t)^3 =
st^3+st^2v+stv^2+sv^3 $
\item[]
$\phantom{b' (w')^2 \mapsto}+t^2 v^3 + t^3 v^2 $
\item[]
$ (f')^*: b' (w')^2 \mapsto (st + t^3 ) (v^2 + t^2) = stv^2 + st^3 +
t^3 v^2 $
\item[] 
$ (f')^*: c' w' \mapsto st^2 (v+t)$
\item[]
$\phantom{b' (w')^2 \mapsto}= st^2 v + s t^3 $

This is an isomorphism.
\item[6)]
$ f^*:b w^3 \mapsto stv^3 $
\item[]
$ f^*: cw^2 \mapsto s t^2 v^2 $
\item[]
$(f')^*:  b' (w')^3 \mapsto  (st + t^3 ) (v+t)^3 = stv^3 + st^2 v^2 +
st^3 v +
 t^3 v^3 $
\item[]
$(f')^*: c'(w')^2\mapsto st^2(v+t)^2=st^2v^2 $

$Ker=<cw^2+c'{w'}^2>$

$Coker= <t^3v^3>$

\item[7)]
$f^*: cw^3 \mapsto st^2 v^3 $
\item[]
$ (f')^*: c' (w')^3 \mapsto s t^2 (v+t)^3 = st^2 v^3 + s t^3 v^2 $

This is an isomorphism.
\item[8)]
$Coker = <st^3v^3>$
\end{itemize}

\smallskip
Therefore in $\tilde{H}^*(\vv{M})$   we have generators
$$
\begin{array}{rcl}
a_2 & \rightarrow& a + w^2 + a' + (w')^2\cr
a_3 & \leftarrow&  vt\cr
a_4&\rightarrow& c + c'\cr
a_5&\leftarrow &v^3 t\cr
a_6& \rightarrow& cw^2 + c' (w')^2\cr
a_7&\leftarrow &t^3 v^3\cr
a_9&\leftarrow& s t^3 v^3\cr
\end{array}
$$

Thus
\begin{theorem} \label{cohomologyM'}
The mod $2$ cohomology of $\vv{M}$ is
$$H^q(\vv{M};\FF)=
\begin{cases}
\FF&q=0,2,3,4,5,6,7,9;\cr
0&\mbox{otherwise}.\cr
\end{cases}
$$
\end{theorem}

Now consider the sequence with integer coefficients.

As earlier,
 $f$ denotes $j\times 1_\RP$ while
$f'$ denotes $(j'\times 1_\RP)\circ\tilde{\tau}$.

$$\displaylines{
H^1 (\vv{M}) \rTo
\underbrace{H^1 (\hat{\vv{U}})}_0 \oplus \underbrace{H^1
(\hat{\vv{V}})}_0
\rTo^{\bigl(f \bot f'\bigr)^*}
\underbrace{H^1 (\hat{B})}_0\rTo\cr
H^2 (\vv{M}) \rTo
\underbrace{H^2(\hat{\vv{U}})}_{\ZZ+\FFF} \oplus
\underbrace{H^2 (\hat{\vv{V}})}_{\ZZ+\FFF} \rTo^{\bigl(f \bot
f'\bigr)^*}
\underbrace{H^2 (\hat{B})}_{\ZZ+\FFF+\FFF}\rTo \cr
H^3 (\vv{M}) \rTo
\underbrace{H^3 (\hat{\vv{U}})}_{\ZZ} \oplus
\underbrace{H^3 (\hat{\vv{V}})}_{\ZZ}
\rTo^{\bigl(f \bot f'\bigr)^*} \underbrace{H^3
(\hat{B})}_{\ZZ+\ZZ+\FFF}\rTo \cr
H^4 (\vv{M}) \rTo
\underbrace{H^4 (\hat{\vv{U}})}_{\FFF+\FFF} \oplus
\underbrace{H^4 (\hat{\vv{V}})}_{\FFF+\FFF}
\rTo^{\bigl(f \bot f'\bigr)^*} \underbrace{H^4 (\hat{B})}_{(\FFF)^3}\rTo
\cr
H^5 (\hat{M}) \rTo
\underbrace{H^5 (\hat{\vv{U}})}_{\ZZ+\FFF} \oplus
\underbrace{H^5 (\hat{\vv{V}})}_{\ZZ+\FFF}
\rTo^{\bigl(f \bot f'\bigr)^*} \underbrace{H^5
(\hat{B})}_{\ZZ+\ZZ+(\FFF)^3}\rTo \cr
H^6 (\vv{M}) \rTo
\underbrace{H^6 (\hat{\vv{U}})}_{\FFF} \oplus
\underbrace{H^6 (\hat{\vv{V}})}_{\FFF}
\rTo^{\bigl(f \bot f'\bigr)^*} \underbrace{H^6 (\hat{B})}_{\ZZ+\FFF}\rTo
\cr
H^7 (\vv{M}) \rTo
\underbrace{H^7 (\hat{\vv{U}})}_{\FFF} \oplus
\underbrace{ H^7 (\hat{\vv{V}})}_{\FFF}
\rTo^{\bigl(f \bot f'\bigr)^*} \underbrace{H^7
(\hat{B})}_{\FFF+\FFF}\rTo \cr
H^8 (\vv{M}) \rTo
\underbrace{H^8 (\hat{\vv{U}})}_0 \oplus \underbrace{H^8
(\hat{\vv{V}})}_0
\rTo^{\bigl(f \bot f'\bigr)^*} \underbrace{H^8 (\hat{B})}_\ZZ\rTo \cr
H^9 (\vv{M}) \rTo 0.\cr
}$$

The segment
$0\to \coker\delta\to H^4(\vv{M})\to \ker\delta\to0 $
becomes $$0\to \ZZ/2\to H^4(\vv{M})\to \ZZ/2\to0, $$
so we see that $H^4(\vv{M})$ has $4$~elements.
Our calculation with $\FF$~coefficients shows that the
mod~$2$ reduction $H^4(\vv{M};\FF)$ of $H^4(\vv{M})$ has
a single summand.
Thus $H^4(\vv{M})=\ZZ/4$.
Similarly $H^6(\vv{M})=\ZZ/4$.
Therefore the integral cohomology of $\vv{M}$ is
\begin{theorem} \label{zcohM'}
$$H^q(\vv{M})=
\begin{cases}
\ZZ&q=0,2,7,9;\cr
\ZZ/4&q=4,6;\cr
0&\mbox{otherwise}.\cr
\end{cases}
$$
\end{theorem}

\medskip
We conclude this section by computing the ring structure
of~$H^*(\vv{M};\FF)$.
Using $\FF$ coefficients,
as a group, $H^*(\vv{M})=\langle 1,a_2,a_3,a_4,a_5,a_6,a_7, a_9\rangle.$

Let $q: M \to A$ be the bundle projection and
let $\iota:S^2 \to A$ denote the inclusion of the lowest degree cell of
$A$.
We may choose $\iota$ to be smooth.
Let $N \to S^2$ be the restriction of the bundle projection $q$ to
$S^2$.
 Then $N$ is a 5-manifold and from the Serre spectral sequence we
 calculate
$H^*(N) =
\langle 1,a'_2,a'_3,a'_5\rangle.$
By Poincar\'e duality $a'_2 a'_3 = a'_5. $
Therefore naturality shows that $a_2 a_3 = a_5 $ in $H^*(\vv{M})$.

Poincar\'e duality gives $a_4 a_5 = a_9 $.
In other words,  $a_2 a_3 a_4 = a_9. $	
In particular $a_2 a_4 \ne 0 $ and $a_3 a_4 \ne 0 $.
Thus $a_2 a_4 = a_6$ and $a_3 a_4 = a_7$.
By Poincar\'e duality we also have $a_2 a_7 = a_9$ and $a_3 a_6 = a_9$.
This describes all the cup products except for
$(a_2)^2$ and $(a_3)^2$.

Next we show that $(a_2)^2=0$.
For a $CW$-complex $Y$, let $Y^{(k)}$ denote its $k$-skeleton.
Let $\kappa_{\vv{M}}:S^3\to\vv{M}^{(3)}$ denote the attaching map which
produces that $4$-skeleton of $\vv{M}$ and let
$\kappa_{\vv{A}}:S^3\to\vv{A}^{(3)}$ denote the attaching map which
produces
that $4$-skeleton of~$\vv{A}$.
By naturality we have a diagram
\begin{diagram}
S^3&\rTo^{\kappa_M}&\vv{M}^{(3)}&\rTo&\vv{M}^{(4)}\cr
\dEqualto&&\dTo_q&&\dTo_q\cr
S^3&\rTo^{\kappa_A}&\vv{A}^{(3)}&\rTo&\vv{A}^{(4)}\cr
\end{diagram}

The $3$-skeleton of $\vv{M}$ is
${\vv{M}}^{(3)}=S^2\wwedge S^3$.
The group structure of $H^*(\vv{M})$ shows that the attaching map
has the form
$\kappa_{\vv{M}}=\lambda\eta\oplus 4\iota_3\in
\pi_3(S^2)\oplus\pi_3(S^3)\cong\ZZ\oplus\ZZ$
for some integer~$\lambda$, where $\iota_3$ is the identity map
on~$S^3$.

The $3$-skeleton of $\vv{A}$ is
${\vv{A}}^{(3)}=S^2$.
The ring structure of $H^*(\vv{A})$ shows that the attaching map is
$\kappa_{A'}=4\eta\in\pi_3(S^2)\cong \ZZ$.

The long exact homotopy sequence of the bundle $\RP\to\vv{M}\to\vv{A}$
implies that
$\pi_2(\vv{M})\to\pi_2(\vv{A})$ is multiplication by $2$,
so ${\vv{M}}^{(3)}\to {\vv{A}}^{(3)}$ has the form $2\iota_2\bot\mu\eta$
for some integer~$\mu$.
Thus the diagram gives
$$4\eta=\kappa_{\vv{A}}=q\circ\kappa_{\vv{M}}=(2\lambda+4\mu)\eta.$$
It follows that $\lambda$ is even and hence $(a_2)^2=0$.

Finally we show that $(a_3)^2=0$.
A property of Steenrod operations is that $\Sq^k(x)=x^2$ if $|x|=k$.
Thus $(a_3)^2=\Sq^3(a_3)=\Sq^1\Sq^2(a_3)$ from the Adem relations.

The Mayer-Vietoris sequence for $\vv{M}$ gives $a_3=\delta(vt)$,
where $\delta$ is the connecting homomorphism.
Therefore
$\Sq^2(a_3)=\delta\bigl(\Sq^2(vt)\bigr)=\delta(\Sq^2vt+\Sq^1v\Sq^1t+v\Sq^2t)
=\delta(0+v^2t^2+0)=0,$ since
$v^2t^2=\im\bigl((j\times 1_\RP)\bot(j'\times
1_\RP)\circ\bar{\tau}\bigr)^*.$
Hence $(a_3)^2=0$.

\section{$6$-manifolds}\label{wall}

According to Wall \cite{Wall}, any simply connected $6$-manifold~$N$
is diffeomorphic to the connected sum
$(S^3\times S^3)^{\#r}\#\vv{N}$
for some integer~$r$ and some simply connected $6$-manifold $\vv{N}$
with
$H^3(\vv{N})=0$.

Further, Wall shows that simply connected $6$-manifolds $\vv{N}$ with
$H^2(\vv{N})\cong\ZZ$ and $H^3(\vv{N})=0$ are determined by:
\begin{enumerate}
\item
a positive integer $d$ given by $x^3=d z$
where $x$ generates $H^2(\vv{N})$ and $z$ generates $H^6(N) $
(which determines the cup products in~$\vv{N}$);
\item
an integer $p$ determined by $xp_1(\vv{N})=pz$, where $p_1(\vv{N})\in
H^4(\vv{N})$
is the first Pontrjagin class of the tangent bundle of~$\vv{N}$.
The sign of $p$ is well defined, because if another
generator $-x$ is chosen, then $z$ is replaced by $-z$ and the minus
signs cancel.
\end{enumerate}
He shows that they must satisfy $p\equiv 4d$ modulo~$24$.

In our case, we know that $A$ is a simply connected $6$-manifold (see
Theorem \ref{thm:cohomologya}), and by the discussion in subsection
\ref{sec:cohomologya} we have $A=(S^3\times S^3)\# (S^3\times
S^3)\#\vv{A}$. With the ring structure given in Theorem
\ref{ringstructure}, we know that $d=4$.

For a differentable manifold~$N$, let $\mathcal{T}N$ denote the tangent
bundle of~$N$.
According to Wall, our space has $c_1\bigl(\mathcal{T}A \bigr)=2x$,
$c_2\bigl(\mathcal{T}A \bigr)=12y$, where $y$ generates $H^4(A)$ and
$xy=z$. So
$xp_1\bigl(A \bigr)=x(c_1^2-2c_2)=4x^3-24xy=16z-24z =-8z$.
Thus $p=-8$.

By Wall's classification theorem, $A$ is uniquely determined as a
simply connected differentiable $6$-manifold by these data.

Note that Wall's classification theorem is more general.
The above statement is the restriction to our situation. Wall's results
were later extended by Jupp \cite{Jupp} and Zhubr \cite{Zubr}.

\section{Cohomology of $M$}\label{cohM}
{\ }

In this section we use the results of~\cite{JS}
together with our earlier calculation of $H^*(\vv{M})$ to obtain
the cohomology of~$M$.

According to Wall, we can write $ A = K\# \vv{A}$ where
$$K = (S^3\times S^3)\#(S^3\times S^3). $$
Consider the bundle $\RP\to \vv{M}\to\vv{A}$.
Set $K':=K\setminus\{\rm chart\}\simeq\wwedge_4 S^3$.
According to the Decomposition Theorem of \cite{JS}, we have
$$H^q(M)\cong H^q(\vv{M})\oplus
H^q\bigl((K'\times \RP)/(\pt \times \RP)\bigr)$$
for~$0<q<9$.
Since
$(K'\times \RP)/(\pt \times \RP)\simeq\wwedge_4 (S^3\wwedge \Sigma^3
\RP)$
we find that
$$H^q(M)=
\begin{cases}
\ZZ&q=0,2,7,9;\cr
\ZZ^4&q=3;\cr
\ZZ/4&q=4;\cr
(\ZZ/2)^4&q=5;\cr
\ZZ^4\oplus \ZZ/4&q=6;\cr
0&\mbox{otherwise}.\cr
\end{cases}
$$

We conclude by using the same method to calculate the cohomology of
the total space $E(L_A)$ of the prequantum line bundle $L_A$ over~$A$.
Recall that $L_A$ is the pullback to $A$ of the line
bundle~$L$ discussed in~\S\ref{qAtiyah}.
Applying the theorem to $S^1\to{\qb{A}}\to \vv{A}$ gives
$$H^q\big(E(L_A)\bigr)\cong H^q(\qb{A})\oplus
H^q\bigl((K'\times S^1)/(\pt \times S^1)\bigr)$$
for~$0<q<7$.
Since
$(K'\times S^1)/(\pt \times S^1)\simeq\wwedge_4 (S^3\wwedge \Sigma^3
S^1)$
we find that
$$H^q\bigl(E(L_A)\bigr)=
\begin{cases}
\ZZ&q=0,7;\cr
\ZZ^4&q=3;\cr
\ZZ^4\oplus \ZZ/4&q=4;\cr
0&\mbox{otherwise}.\cr
\end{cases}
$$

\end{document}